\documentclass[a4paper,11pt]{article}
\usepackage[margin=1in]{geometry}

\usepackage{amsthm,amsmath,amssymb,amsfonts}

\usepackage[bookmarks=true, colorlinks=true, linkcolor=black, citecolor=black, urlcolor=black, plainpages=false, pdfpagelabels,final]{hyperref}
\usepackage{cleveref}
\usepackage{xcolor}
\usepackage{colonequals}

\usepackage[showdeletions]{color-edits}
\definecolor{darkcyan}{rgb}{0.0, 0.55, 0.55}
\addauthor[Melanie]{mg}{Peach}
\addauthor[EH]{eh}{darkcyan}

\usepackage[backend=bibtex,style=alphabetic,giveninits=true,doi=false,isbn=false,url=false,sorting=nyt,maxbibnames=99,date=year,maxalphanames=999]{biblatex}
\renewbibmacro{in:}{} 
\def \RR {\mathbb{R}}
\def \eps {\varepsilon}
\def \loc {\mathrm{loc}}
\def \Wloc {W_{\mathrm{loc}}}
\def \Lloc {L_{\mathrm{loc}}}
\def \Cloc {C_\mathrm{loc}}
\def \Vol {\mathrm{Vol}}
\def \D {\mathcal{D}}
\def \Ric {\mathrm{Ric}}
\def \Riem {\mathrm{Riem}}
\def \dgeps {d\mu_{g_\varepsilon}}
\def \dg {d\mu_{g}}
\def \la {\langle}
\def \ra {\rangle}
\def \gtilde{\widetilde{g}}
\def \RCD {{\sf RCD}}

\DeclareMathOperator \supp {supp}
\DeclareMathOperator \tr {tr}

\newcommand\restr[2]{{
  \left.\kern-\nulldelimiterspace 
  #1 
  \vphantom{\big|} 
  \right|_{#2} 
  }}

\theoremstyle{plain}
\newtheorem{theorem}{Theorem}[section]
\numberwithin{equation}{section}

\newtheorem{definition}[theorem]{Definition}
\newtheorem{proposition}[theorem]{Proposition}
\newtheorem{corollary}[theorem]{Corollary} 
\newtheorem{lemma}[theorem]{Lemma}
\newtheorem{remark}[theorem]{Remark}

\crefname{lemma}{Lemma}{Lemmas}
\crefname{equation}{Equation}{Equations}

\begin{document}

\title{A low-regularity Riemannian positive mass theorem for non-spin manifolds with distributional curvature}

\author{
Eduardo Hafemann
\thanks{Fachbereich Mathematik, Universit\"at Hamburg, Bundesstraße 55, 20146 Hamburg, Germany. Email: \href{mailto:eduardo.hafemann@uni-hamburg.de}{\tt eduardo.hafemann@uni-hamburg.de}}}
\maketitle

\begin{center}
Abstract
\end{center}
This article establishes a low-regularity Riemannian positive mass theorem for non-spin manifolds whose metrics are only $C^0 \cap W_{\mathrm{loc}}^{1,n}$ and smooth outside a compact set. The main theorem asserts that asymptotically flat manifolds with nonnegative distributional scalar curvature have nonnegative ADM mass. The proof uses smooth approximations of the metric together with a Sobolev version of Friedrichs' Lemma, which yields improved convergence for commutators between differentiation and convolution operators. Rigidity in the metric-space sense is obtained via the volume comparison theory of $\sf{RCD}$-spaces after establishing that zero ADM mass implies Ricci flatness, which fundamentally relies on these improved estimates. In essence, a version of the main theorem of Lee-LeFloch \cite{LeeLeFloch2015} is presented in which the spin condition is removed under the assumption that the metric is smooth outside a compact set.


\tableofcontents

\clearpage

\section{Introduction}

The celebrated theorem known as the Riemannian positive mass theorem stands as one of the fundamental theorems in mathematical relativity:
\begin{theorem} \label{theo:OriginalPMT}
Let $(M^n, g)$ be a complete, smooth asymptotically flat manifold, and suppose that the scalar curvature of $g$ is integrable and nonnegative. Then the ADM mass of $(M^n, g)$ is nonnegative. Moreover, the ADM mass is zero if and only if $(M^n, g)$ is isometric to the Euclidean space $(\RR^n, \delta)$.
\end{theorem}
This theorem was first proved by Schoen and Yau \cite{PMTSchoenYau} for $3\leq n \leq 7$, and the higher-dimensional cases have been treated in \cite{Schoen2019,Lohkamp2016}, and more recently in \cite{ChodoshMantoulidisSchulze2023, ChodoshMantoulidisSchulzeWang2025,BiHaoHeShiZhu26, BrendleWang2026}. Meanwhile, Witten \cite{Witten1981} extended the result to higher dimensions for spin manifolds. Bartnik \cite{Bartnik1986} showed that the ADM mass is a geometric invariant, i.e., independent of the choice of coordinates at infinity if the metric is $\Wloc^{2,p}$, with $p>n$. It is worth mentioning that Witten's spinor argument works under this metric regularity as well. 

The positive mass theorem, particularly its rigidity case, can be seen as a result that links geometry and topology via scalar curvature lower bounds. Although scalar curvature is classically defined using second derivatives of the metric, there is evidence that scalar curvature lower bounds may be a $C^0$ property of the metric. This idea is supported by a result of Gromov~\cite{Gromov2014} who, using a newly proposed polyhedral comparison theory for positive scalar curvature, showed that scalar curvature lower bounds are closed under $C^0$ limits of $C^2$ metrics, provided the limiting metric is $C^2$ (see \cite{Bamler2016} for a Ricci-flow proof). This suggests that, provided a suitable weak notion of scalar curvature and an appropriate notion of mass, a $C^0$ positive mass theorem may hold, that is, a weak curvature condition for $C^0$ metrics may imply the nonnegativity of a geometric invariant (mass), with vanishing mass implying isometry to Euclidean space. This problem is also relevant to mathematical relativity, since, for time-symmetric initial data sets satisfying the Einstein constraint equations, the dominant energy condition reduces to nonnegative scalar curvature, and studying this problem may reveal how energy conditions are encoded at the $C^0$ level of the metric (see \cite{CavallettiMondino2022} for synthetic energy conditions involving Ricci curvature lower bounds compatible with $C^0$ metrics). Therefore, it is geometrically and physically interesting to inquire whether the assumptions on the metric can be relaxed and whether the classical positive mass theorem can be extended to nonsmooth metrics. A first step in this investigation is to consider metrics for which scalar curvature can be defined only as a distribution.

In this regard, several results concerning the positive mass theorem for nonsmooth metrics have been established. For manifolds with corners, Miao \cite{Miao2002} showed that the mass is nonnegative by smoothing the metrics and then deforming them conformally. The rigidity case was resolved by McFeron and Székelyhidi \cite{McFeron2012-gs} using Ricci flow techniques. Lower-dimensional singular sets are also considered by Lee \cite{Lee2013} through the same conformal deformation method. 

At even lower regularity, Lee-LeFloch \cite{LeeLeFloch2015} introduced a distributional definition of scalar curvature and ADM mass suitable for $C^0 \cap \Wloc^{1,n}$ metrics, and showed that Witten's argument can be extended in this setting for $n$-dimensional spin manifolds. It would be desirable to obtain an analogous theorem for non-spin manifolds with the same regularity. Several steps toward such a result have been made, for instance, see \cite{grant2014positive, Li2020, Jiang2022}. Grant and Tassotti \cite{grant2014positive} showed the nonnegativity of the ADM mass for non-spin manifolds with $C^0 \cap \Wloc^{2, n/2}$ metrics that are smooth away from a compact set. Using Ricci flow, Li \cite{Li2020} obtained a non-spin result for $C^0 \cap \Wloc^{1,p}$, $p>n$, metrics under a bounded curvature condition. However, the bounded curvature condition already implies a higher metric regularity, namely, $g \in \Wloc^{2,p}$, $p>n$ (see \cite[Theorem 3.1]{Li2020}). Jiang, Sheng and Zhang \cite{Jiang2022} improved the regularity to $C^0 \cap \Wloc^{1,p}$, $p \in [n, \infty]$, while assuming that the metric is smooth away from a singular set $\Sigma$ with Hausdorff measure $\mathcal{H}^{n-1}(\Sigma)=0$ if $p=\infty$ or $\mathcal{H}^{n-\frac{p}{p-1}}(\Sigma)$ is finite if $p<\infty$. This result can be viewed as an extension of previous results \cite{Lee2013, LeeLeFloch2015, Shi2018}.

This paper focuses on a low-regularity version of the positive mass theorem for $C^0 \cap \Wloc^{1, n}$ metrics that are smooth outside a compact set $K\subset M$, without the spin assumption or the additional conditions on the singular set imposed in previous works. This setting is physically natural because it allows the nonsmooth behaviour of the metric to be confined to a compact region while retaining a smooth asymptotically flat end, on which the ADM mass is still well-defined classically. For example, related rough $3$-dimensional asymptotically flat vacuum initial data were constructed by Maxwell \cite{Maxwell2006}, with metrics in $H^s_\loc$, $s>3/2$. In particular, this regularity is stronger than $C^0 \cap \Wloc^{1, 3}$ but does not imply Lipschitz continuity when $s<5/2$. The main result is stated as follows. 

\begin{theorem}\label{theo:PMT}
    Let $M^n$, $n\geq 3$, be a smooth manifold endowed with a complete, asymptotically flat Riemannian metric $g \in C^0\cap \Wloc^{1,n}$, smooth outside a compact set. If the distributional scalar curvature $R[g]$ is nonnegative, then the ADM mass $m(g)$ is nonnegative. Moreover, $m(g) = 0$ if and only if $(M^n, d_g)$ is isometric to $(\RR^n, d_\delta)$ as a metric space. If, in addition, $g \in \Wloc^{1,p}$, $p>n$, then there exists a $C^{1,1-\frac{n}{p}}$ isometry between $(M, g)$ and $(\RR^n, \delta)$. 
\end{theorem}

It follows from the Sobolev embedding theorem that $\Wloc^{1,p} \subset C^0 \cap \Wloc^{1,n}$, for any $p>n$. Hence, the nonnegativity of the ADM mass in \Cref{theo:PMT} holds for $\Wloc^{1,p}$, $p>n$. For the Riemannian rigidity statement, we further require $p>n$, since the volume comparison argument used in its proof yields only a metric isometry, which is then promoted to a Riemannian isometry by the low-regularity Myers-Steenrod theorem \cite[Corollary C]{Matveev2017}, which requires $p>n$. The definition of distributional curvature follows the distribution theory on smooth manifolds according to \cite{LeFloch2007,Grosser2001, Gra:20}. Although the frameworks in \cite{LeeLeFloch2015} and \cite{Jiang2022} differ from the one adopted here, the definitions are equivalent, as discussed in \Cref{rmk:DistDef}.

As a consequence, for metrics that are smooth away from a compact set, we recover the theorem of Lee-LeFloch while removing the spin assumption. In particular, the result includes earlier theorems for 'piecewise regular' metrics by Miao \cite{Miao2002}, Shi and Tam \cite{Shi2002}, McFeron and Sz{\'e}kelyhidi \cite{McFeron2012-gs}, Lee \cite{Lee2013}, and Jiang, Sheng, and Zhang \cite{Jiang2022}. Furthermore, by the Sobolev embedding, the result also generalises Grant and Tassotti's result \cite{grant2014positive} for $\Wloc^{2, n/2}$ metrics that are $C^2$ and asymptotically flat outside a compact set. The results of \cite{Miao2002, Shi2002, McFeron2012-gs, Lee2013} follow directly from Propositions 5.1 and 5.2 in \cite{LeeLeFloch2015}, while the result of \cite{Jiang2022} follows from Lemma 2.7 in \cite{Jiang2022}.

\par \medskip

\noindent{}\textbf{Outline of Proof.} The proof of the main theorem is divided into two parts: the nonnegativity of the ADM mass and the rigidity case. For the nonnegativity, the argument relies on approximation techniques, Friedrichs' Lemma \cite{Gra:20, Calisti2025}, and a conformal deformation approach \cite{Miao2002, grant2014positive}. The rigidity part, in turn, builds on previously established results and comparison theory for $\RCD$-spaces (see \cite{Jiang2022, Gigli2018, mondino2025}). In what follows, we detail these arguments.

The conformal method consists of performing a conformal rescaling of a smooth approximating metric $g_\eps$ of $g$ in order to eliminate the negative part of its scalar curvature. Combining this procedure with the smooth positive mass theorem yields a sequence of metrics $\widetilde{g}_\eps$ with nonnegative ADM mass, satisfying $m(\widetilde{g}_\eps) \to m(g)$ as $\eps \to 0$. A key requirement of this method is a uniform $\Lloc^{n/2}$-smallness bound on the negative part of the scalar curvature of $g_\eps$. This can be ensured in several ways, e.g., if $g$ has regularity $C^0 \cap \Wloc^{2,n/2}$, then $R[g] \in \Lloc^{n/2}$ and one can follow the approach described in \cite{grant2014positive}. In contrast, for $C^0 \cap \Wloc^{1,n}$ metrics, the scalar curvature is only defined distributionally, and a different strategy is required.

A key tool for adapting the approach of \cite{grant2014positive} to $C^0 \cap \Wloc^{1,n}$ metrics is Friedrichs' Lemma (see \Cref{prop:FriedrichsWp}). This lemma provides improved convergence for commutators between differentiation and convolution operators. More precisely, let $a \in \Wloc^{1,p}\left(\RR^n\right)$ and $f \in \Lloc^p\left(\RR^n\right)$, with $p\geq2$, and let $\rho_\eps$ be the standard mollifier. Then
\begin{equation} \label{eq:intro_Fried}
    \|\left(a * \rho_{\eps}\right)\left(f * \rho_{\eps}\right)-(a f) * \rho_{\eps}\| _{W^{1,p/2}(K)} \longrightarrow 0 \text{ as } \eps \to 0,
\end{equation}
 for any compact $K \subset \RR^n$ (see \Cref{prop:FriedrichsWp} and \Cref{cor:ConvFriedPlusApproximations}).
 
 To see why this is useful, recall that the highest-order terms in the scalar curvature have the schematic form $\partial ( g^{-1} \partial g)$. If $g \in C^0 \cap \Wloc^{1,n}$, then $g^{-1} \partial g$ has precisely the structure $af$ appearing above. There are two natural regularisations to compare: the scalar curvature $R[g_\eps]$ of a mollified metric $g_\eps=g*\rho_\eps$, which roughly corresponds to $(a*\rho_\eps)(f*\rho_\eps)$, and the regularisation of the distributional scalar curvature $R[g]$ given by $R[g] *\rho_\eps$, corresponding to $(af)*\rho_\eps$. Although each of these approximations individually need only converge distributionally, Friedrichs' lemma, together with the convergence of the lower-order terms, shows that their difference converges one derivative better than one would initially expect, i.e., $\Lloc^{n/2}$ instead of $\D^\prime$. This provides the missing ingredient for controlling the $\Lloc^{n/2}$ norm of the negative part of $R[g_\eps]$, provided that $R[g]$ is a nonnegative distribution (see \Cref{prop:NegScalConv}), thereby allowing us to apply the conformal method. Notably, regularisation approaches based on Friedrichs' lemma have played a central role in low-regularity results in Lorentzian geometry, including singularity theorems \cite{KSSV:15,KSV:15,GGKS:18,Gra:20,KOSS:22,Calisti2025}, and now we apply it to low-regularity Riemannian geometry. Building on \cite{Calisti2025}, the lemma is here extended from $C^{0,1}$ to $C^0 \cap \Wloc^{1,p}$, $p\geq2$. 

Regarding the rigidity case, the main idea is to employ the volume comparison theory of $\RCD$ spaces to show that if an asymptotically flat manifold $(M, g)$ has nonnegative Ricci curvature in the $\RCD$ sense, then it must be isometric, as a metric space, to Euclidean space, as argued in \cite{Jiang2022}. Assuming additionally that $g \in W^{1,p}_\loc$, $p>n$, the isometry between metric spaces can then be promoted to an isometry between Riemannian manifolds via the low-regularity version of the Myers-Steenrod theorem \cite[Corollary C]{Matveev2017}. The main challenge, however, lies in proving that if the ADM mass is zero, then $(M, g)$ has nonnegative Ricci curvature in the $\RCD$ sense.

Interestingly, Friedrichs' lemma and \Cref{prop:ConvofRs} are the cornerstones of the argument used to address this challenge. The main difficulty lies in the fact that the standard rigidity argument cannot be applied directly to $g$, but must instead be applied to a family of conformally rescaled smooth metrics $\widetilde{g}_\eps$ with nonnegative scalar curvature, whose ADM masses need not to vanish. As in the smooth case, the idea is to show that if the distributional curvature (scalar or Ricci) is nontrivial, then $m(\widetilde{g}_\eps)<0$ for $\eps$ sufficiently small, thereby contradicting the smooth positive mass theorem. The estimates needed to prove scalar flatness rely on a combination of \Cref{prop:NegScalConv}, which follows from Friedrichs' lemma, and the technical estimate in \Cref{lemma:ScalarLowerBound}. The proof of Ricci flatness follows a similar procedure and the general argument of \cite[Lemma 4.1]{Jiang2022}, but establishes Ricci flatness globally rather than only outside a compact set, due to \Cref{prop:ConvofRs} and a careful analysis of the Laplacian. Consequently, we conclude that $(M, g)$ is Ricci-flat distributionally. By \cite[Theorem 7.2]{mondino2025}, it then has nonnegative Ricci curvature in $\RCD$ sense.
\par \medskip
\noindent{}\textbf{Organisation of the paper.} The definitions of asymptotic flatness and ADM mass are given in \Cref{sec:preliminaries}. We also review distributions on manifolds to define the distributional scalar curvature for metrics of class $C^0 \cap \Wloc^{1,2}$. \Cref{section:ScalarCurvature} presents the main tools and convergence results that allow us to obtain the main result at lower regularity than in \cite{grant2014positive} and to remove the spin assumption from \cite{LeeLeFloch2015} for metrics that are smooth outside a compact set. The proof of the main theorem is presented in \Cref{section:PMT}.
\par \medskip
\noindent{} \textbf{Acknowledgments.}\label{acknowledgments} 
I thank my supervisor, Melanie Graf, for suggesting this topic and for her patient guidance during our many discussions. Her support was essential to this work. I also acknowledge support from the Deutsche Forschungsgemeinschaft (DFG, German Research Foundation) under Germany’s Excellence Strategy – EXC 2121 "Quantum Universe" – 390833306.

\section{Preliminaries} \label{sec:preliminaries}

\subsection{ADM mass}

Let $g$ be a continuous Riemannian metric on a smooth manifold $M^n$ of dimension $n\geq 3$. The pair $(M, g)$ is said to be \textit{asymptotically flat} if there exists a compact set $K\subset M$ such that $M \setminus K$ is diffeomorphic to $\RR^n \setminus \overline{B_{1}(0)}$, i.e., there exists a diffeomorphism 
$$\Phi:  M \setminus K \to \RR^n \setminus \overline{B_1(0)},$$
where $\overline{B_1(0)}$ is the standard closed unit ball\footnote{For simplicity, we assume that the manifold has only one end $M \setminus K$. The generalisation to multiple ends is a standard argument.}. Moreover, if we think $\Phi$ as a coordinate chart with coordinates $x^1,\ldots, x^n$, then, in this coordinate chart, we assume that $g_{ij}$ is $C^2$ and satisfies
    \begin{align*}
        g_{ij} - \delta_{ij} & =  O_2(|x|^{-\tau}),
    \end{align*}
for some $\tau>(n-2)/2$. Here, for a $C^k$ function $f$, we write $f = O_k(|x|^{-\tau})$ if there exist constants $C, R>0$ such that $|f(x)| + |\partial f(x)| |x| + \ldots + |\partial^k f(x)| |x|^{k}\leq C |x|^{-\tau}$, for $|x|\geq R$,. We also assume that $R[g] \in L^1(M \setminus K)$.

\begin{remark}
For $C^0$ metrics, the Christoffel symbols are no longer defined pointwise, as a consequence, the existence theorem for geodesics does not apply, and Hopf-Rinow no longer holds. However, if $g$ is continuous, $(M, g)$ can be viewed as a metric space with a distance function $d_g$ induced by this metric. Therefore, we say $(M, g)$ is \textit{complete} if $(M, d_g)$ is a complete metric space.
\end{remark}

As an asymptotically flat manifold $(M, g)$ possesses a metric that is smooth at infinity, its scalar curvature $R[g]$ is well-defined as $|x| \to \infty$, in particular, $R[g] \in L^1(M\setminus K)$. Then, the \textit{ADM mass of $(M, g)$} (see \cite{ADMCitation,Bartnik1986}) is defined by
\begin{equation}
m(g) = \lim_{r \to \infty} \frac{1}{2(n-1)\omega_{n-1}} \int_{S_r} \sum_{i,j=1}^n (g_{ij,i}-g_{ii,j})\nu_j\,dS,
\end{equation}
where $\omega_{n-1}$ is the area of the standard unit $(n-1)$-sphere in $\RR^n$, $S_r$ is the coordinate sphere in $M \setminus K$ of radius $r$, $\nu$ is its outward unit normal, and $dS$ is the Euclidean area element on $S_r$.

\subsection{Convergence and Sobolev spaces}
Fix a complete smooth Riemannian background metric $h$ on $M$. The following discussion is independent of $h$, since on compact sets, the induced norms are equivalent for all $h$. For any continuous $(r, s)$-tensor field $\mathcal{T}$ on $M$ and any subset $A \subset M$ we define
\begin{align*}
\|\mathcal{T}\|_{\infty, A} := &  \sup \{ \big|\mathcal{T}_x(\theta^1,\ldots, \theta^r, X_1,\ldots, X_s) \big| : x \in A, \\ 
& \theta^j \in T_x^*M, \|\theta^j\|_h =1, X_i \in  T_xM, \|X_i\|_h = 1 \}.   
\end{align*}
Given $K \subset M$ a compact, it follows that a net $\{\mathcal{T}_\eps\}$ of continuous $(r, s)$-tensor fields converges to $\mathcal{T}$ with respect to $\|\cdot\|_{\infty, K}$ if and only if 
\begin{equation}
\big({\mathcal{T}_\eps}\big)^{j_1\ldots j_r}_{i_1 \ldots i_s} \to \big(\mathcal{T}\big)^{j_1\ldots j_r}_{i_1 \ldots i_s} \text{ uniformly on } U\cap K,    
\end{equation}
for any coordinate chart $(U, (x^1,\ldots,x^n))$. If $\mathcal{T}_\eps \to \mathcal{T}$ with respect to $\| \cdot \|_{\infty, K}$ for any compact $K \subset M$, we say that $\mathcal{T}_\eps \to \mathcal{T}$ locally uniformly (or in $\Cloc^0$). Analogously, if $\mathcal{T}$ and $\mathcal{T}_\eps$ are $C^k$, $k\geq1$, we say that $\mathcal{T}_\eps \to \mathcal{T}$ in $\Cloc^k$ if both $\mathcal{T}_\eps \to \mathcal{T}$ and $(\nabla^h)^i \mathcal{T}_\eps \to (\nabla^h)^i \mathcal{T}$ in $\Cloc^0$ for all $1 \leq i \leq k$, where $\nabla^h$ denotes the Levi-Civita connection of the background metric $h$. Note that $(\nabla^h)^i \mathcal{T}_\eps \to (\nabla^h)^i \mathcal{T}$ in $C^0_{\mathrm{loc}}$ is equivalent to all order $k$ partial derivatives of $\mathcal{T}_\eps$ converging locally uniformly converge to their respective counterparts of $\mathcal{T}$ for any chart, along with $\mathcal{T}_\eps \to \mathcal{T}$ in $\Cloc^0$.

For any $(r, s)$-tensor field $\mathcal{T}$ on $M$, any subset $A \subseteq M$, and $p \in [1, \infty]$, we define the $L^p$ norm of $\mathcal{T}$ on $A \subseteq M$ to be
\begin{equation}
\|\mathcal{T}\|_{L^p(A, h)}:= \left(\int_A \|\mathcal{T}\|_h^p\,d\mu_h \right)^{1/p},    
\end{equation}
while for $p=\infty$ we have the usual definition. Similarly, for nonnegative integer $k\geq0$ and $p \in [1, \infty)$, the Sobolev norm is given by
\begin{equation}
\|\mathcal{T}\|_{W^{k,p}(A, h)}:= \sum_{|\alpha| \leq k} \left(\int_A \| (\nabla^h)^\alpha \mathcal{T}\|^p_h\,d\mu_h \right)^{1/p},    
\end{equation}
where $\nabla^h$ denotes the Levi-Civita connection of the background metric $h$. A $(r, s)$-tensor field $\mathcal{T}$ is said to be in $\Lloc^pT^r_s(M)$ (resp. $\Wloc^{k,p}T^r_s(M)$) if $\|\mathcal{T}\|_{L^p(K, h)}$ (resp. $\|\mathcal{T}\|_{W^{k, p}(K, h)}$) is finite for all compact sets $K \subset M$.  In this work, the smooth metrics only differ outside of a compact set, and they are equivalent on a compact set, so $L^p$ norms, as well as $W^{k, p}$ norms, are equivalent. We usually omit the reference to the metric in the norm. Note that the $L^p$ norms are equivalent even for $C^0$ metrics, such as $g$.

In coordinate charts, a net of tensors $\{\mathcal{T}_\eps\}$, with $\mathcal{T}_\eps \in \Wloc^{k,p}T^r_s(M)$ for all $\eps>0$, converges in $\Wloc^{k,p}$ to a tensor $\mathcal{T} \in \Wloc^{k,p}T^r_s(M)$, if and only if, for any coordinate chart, the net of functions
\begin{equation}
\big({\mathcal{T}_\eps}\big)^{j_1\ldots j_r}_{i_1 \ldots i_s} \to \big(\mathcal{T}\big)^{j_1\ldots j_r}_{i_1 \ldots i_s} \text{ in the $\Wloc^{k,p}(M)$ topology}.    
\end{equation}
The previous definition is equivalent to the convergence of the components of the tensor $\mathcal{T}$ for any chosen $X_i \in \mathfrak{X}(M)$ and $\theta^j \in \Omega^1(M)$.

\subsection{Distributional curvature}\label{subsection:DistCurvature}

In this section, the general distributional framework for dealing with low-regularity metric tensors and their associated curvature quantities is presented. This framework is required for the subsequent analysis of the distributional scalar curvature and follows primarily the expositions in \cite{KOSS:22} and \cite{LeFloch2007}.

Let $M$ be an $n$-dimensional smooth manifold, and let $\mu$ be a smooth section of the density bundle of $M$. For any coordinate chart $(U, (x^1,\ldots, x^n))$, the restriction of $\mu$ to $U$ can be expressed locally as 
\begin{equation}
\restr{\mu}{U} = f |dx^1 \wedge \ldots \wedge dx^n|,     
\end{equation}
for some smooth function $f \in C^\infty(U)$. 

Let $(U_i, \psi_i)$ be an atlas of $M$, and denote by $\mu^i  |dx^1 \wedge \ldots \wedge dx^n|$ the coordinate expression of a density, where $\mu^i \in C^\infty(U_i)$. We denote by $\Vol(M)$ the volume bundle over $M$. If $\mu$ is compactly supported, i.e., $\mu \in \Gamma_c(M, \Vol(M))$, the integration of $\mu$ over $M$ is defined as
\begin{equation}
\int_M \mu = \sum_i \int_M \eta_i \mu := \sum_i \int_{\psi_i(U_i)} \eta_i \big(\psi_i^{-1}(x)\big) \mu^i \big(\psi_i^{-1}(x)\big)\,dx,    
\end{equation}
for a given subordinate partition of unity $\eta_i$. This definition is independent of the choice of atlas and partition of unity. A density $\mu$ is said to be \textit{nonnegative} (resp. \textit{positive}) \textit{density} if every $\mu^i$ is nonnegative (resp. positive).

The space of \textit{distributions on $M$} is defined as the topological dual of the space of compactly supported densities, that is,
\begin{equation}
\D^\prime(M) := \big(\Gamma_c(M, \Vol(M)) \big)^\prime.
\end{equation}
Note that any locally integrable function $f \in \Lloc^1(M)$ can be canonically associated with a distribution, i.e.,  $f \in \Lloc^1(M) \mapsto f \in \D^\prime(M)$ by the following functional
\begin{equation}
    \mu \mapsto \la f, \mu \ra:= \int_M f \mu.
\end{equation}
Similarly, the space of \textit{distributional $(r, s)$-tensor fields} is defined as
\begin{equation}
\D^\prime \mathcal{T}^r_s(M) := \D^\prime(M, T^r_sM) := \big(\Gamma_c(M, \mathcal{T}^s_r(M) \otimes \Vol(M)) \big)^\prime,
\end{equation}
where $\mathcal{T}^r_s(M)$ is the space of smooth $(r, s)$-tensor fields on $M$. By \cite[Corollary 3.1.15]{Grosser2001}, the tensor distributions can also be described by
\begin{equation}
    \D^\prime \mathcal{T}^r_s(M) \cong L\big(\Omega^1(M)^r \times \mathfrak{X}(M)^s; \D^\prime(M) \big),  
\end{equation}
where the right-hand side denotes the $C^\infty(M)$-module of multilinear maps from $\Omega^1(M)^r \times \mathfrak{X}(M)^s$ to $\D^\prime(M)$.

The space of locally integrable tensor fields $\Lloc^1 T^r_s(M)$ can be embedded into the space $\D^\prime \mathcal{T}^r_s(M)$, that is, $A \in \Lloc^1 T^r_s(M) \mapsto A \in \D^\prime \mathcal{T}^r_s(M)$ via
\begin{equation}
\mu \mapsto \la A(\theta^1,\ldots, \theta^r, X_1, \ldots, X_s), \mu \ra := \int_M A(\theta^1,\ldots, \theta^r, X_1, \ldots, X_s) \mu,    
\end{equation}
for all $\mu \in \Gamma_c(M, \Vol(M))$, where $\theta^1,\ldots, \theta^r \in \Omega^1(M)$, and  $X_1, \ldots, X_s \in \mathfrak{X}(M)$.

A useful fact is that any smooth $(r, s)$-tensor field $A$ can be uniquely extended to accept one distributional argument. For $A \in \mathcal{T}^r_s(M)$, define 
\begin{equation}
\widetilde{A}: \Omega^1(M)^{r} \times \D^\prime \mathcal{T}^1_0(M) \times \mathfrak{X}(M)^{s-1} \to \D^\prime(M), 
\end{equation}
by setting
\begin{equation}\label{eq:Extension1formDist}
\widetilde{A}(\theta^1,\ldots,\theta^r,Y,X_2,\ldots,X_s):= \la Y, A(\theta^1,\ldots,\theta^r, \cdot, X_2,\ldots,X_s) \ra \text{ in } \D^\prime(M), 
\end{equation}
for all $Y \in \D^\prime \mathcal{T}^1_0(M)$, $X_i \in \mathfrak{X}(M)$ and $\theta^j \in \Omega^1(M)$. Given these preliminaries, a notion of a connection operator in the distributional sense can be introduced.
\begin{definition}
A \textit{distributional connection} is an operator $\nabla: \mathfrak{X}(M) \times \mathfrak{X}(M) \to \D^\prime\mathcal{T}_0^1(M)$ satisfying the following linearity and Leibniz properties
\begin{enumerate}
    \item $\nabla$ is $\RR$-bilinear,
    \item $\nabla_{fX}  Y = f  \nabla_XY$,
    \item $\nabla_X fY = (Xf) Y + f \nabla_X Y$,
\end{enumerate}
for all $X, Y \in \mathfrak{X}(M)$ and $f \in C^\infty(M)$.    
\end{definition}

From the above definition, we observe that a curvature tensor cannot be associated with a general distributional connection, since this would require the multiplication of distributions, as seen in the usual coordinate expression. Furthermore, the standard definition involves second-order covariant derivatives, which are not defined in this framework. Therefore, we restrict our focus to a class of less singular connections.

A distributional connection can be extended to act on general smooth $(r, s)$-tensor fields, i.e.,  $\nabla: \mathfrak{X}(M) \times \mathcal{T}^r_s(M) \to \D^\prime \mathcal{T}^r_s(M)$. Here, we exemplify how to extend it to act on 1-forms. We introduce the map $\nabla: \mathfrak{X}(M) \times \Omega^1(M) \to \D^\prime \mathcal{T}^0_1(M)$ by
\begin{equation}
\la \nabla_X \theta, Y \ra : = X\big( \la \theta, Y \ra \big) - \la \theta, \nabla_X Y \ra \text{ in }\D^\prime(M),    
\end{equation}
for any $X, Y \in \mathfrak{X}(M)$ and any $\theta \in \Omega^1(M)$. The last term is well-defined due to \eqref{eq:Extension1formDist}.

A distributional connection $\nabla$ is called an $\Lloc^2$-connection if $\nabla_XY \in {\Lloc^2}T^1_0(M)$ for all $X, Y \in \mathfrak{X}(M)$. This notion is particularly important, since the largest class of connections allowing a stable definition of the curvature tensor in the distributional sense is the $\Lloc^2$ class (see \cite{LeFloch2007}). Based on \cite{LeFloch2007}, the curvature tensor is then defined as follows.

\begin{definition}\label{def:RiemmTensorDist}
The \textit{distributional Riemann curvature tensor} of an $\Lloc^2$-connection $\nabla$ is the map $\Riem:\mathfrak{X}(M)^3 \to \D^\prime \mathcal{T}^1_0(M)$, given by
$$
\begin{aligned}
\la\Riem(X, Y) Z, \theta\ra= & X\left\la\nabla_Y Z, \theta\right\ra-Y\left\la\nabla_X Z, \theta\right\ra \\
& -\left\la\nabla_Y Z, \nabla_X \theta\right\ra+\left\la\nabla_X Z, \nabla_Y \theta\right\ra-\left\la\nabla_{[X, Y]} Z, \theta\right\ra \text{ in }\D^{\prime}(M),
\end{aligned}
$$
for all $X, Y, Z \in \mathfrak{X}(M)$ and $\theta \in \Omega^1(M)$.  
\end{definition}
As shown in \cite{KOSS:22}, the $\Lloc^2$ connection admits a unique extension given by the operator $\nabla:\mathfrak{X}(M) \times \Lloc^2T^1_0(M) \to \D^\prime\mathcal{T}^1_0(M)$ defined by
\begin{equation}
\la \nabla_X Y, \theta \ra = X\big( \la Y, \theta \ra \big ) - \la Y, \nabla_X \theta \ra \text{ in } \D^\prime(M),    
\end{equation}
for $X \in \mathfrak{X}(M)$, $Y \in \Lloc^2\mathcal{T}^1_0(M)$, and $\theta \in \Omega^1(M)$. Note that $\la Y, \nabla_X \theta \ra$ is only defined as a distribution for connections defined on $\Lloc^2$ or a subspace of $\Lloc^2$. 

Interpreting each term in the sense of distributions, the distributional Riemann curvature tensor can also be expressed as
\begin{equation}
    \Riem(X, Y, Z)(\theta) \equiv \la\Riem(X, Y) Z, \theta\ra := \la \nabla_X \nabla_Y Z - \nabla_Y \nabla_X Z- \nabla_{[X, Y]} Z, \theta\ra,
\end{equation}
for all $X, Y, Z \in \mathfrak{X}(M)$ and $\theta \in \Omega^1(M)$.

\begin{definition}
The \textit{distributional Ricci curvature tensor} of an $\Lloc^2$ connection $\nabla$ is the map $\Ric:\mathfrak{X}(M)^2 \to \D^\prime(M)$, given by
\begin{equation}
\Ric(X, Y) := \la \Riem(X, E_i)Y), E^i \ra \in \D^\prime(M),    
\end{equation}
where $E_i$ is an arbitrary local frame in $TM$ and $E^j$ is its corresponding dual frame, for all $X, Y \in \mathfrak{X}(M)$.
\end{definition}

Given a smooth metric $g$ on $M$, the associated Levi-Civita connection $\nabla$ is torsion-free and satisfies $\nabla g = 0$. For non-smooth metrics, a distributional Levi-Civita connection can be defined, which can be associated with a distributional metric in the following sense.
\begin{definition}
    A distributional metric $g$ on $M$ is a symmetric and non-degenerate $(0, 2)$-tensor distribution on $M$, i.e., a map $g: \mathfrak{X}(M) \times \mathfrak{X}(M) \to \D^\prime(M)$ satisfying
\begin{gather*}
    g(X, Y) = g(Y, X),\\
    g(X, Y) = 0 \text{ for all } Y \implies X = 0,    
\end{gather*}
for all $X, Y \in \mathfrak{X}(M)$.
\end{definition}
In order to obtain a notion of Levi-Civita connection for a distributional metric, we must emulate the torsion-free condition and $\nabla g=0$ in the low-regularity setting. We begin with the following observation for smooth metrics: the musical isomorphism $\flat: X  \in \mathfrak{X}(M) \mapsto X^\flat \in \Omega^1(M)$ is defined by
\begin{equation}
\la X^\flat, Y\ra := g(X, Y),    
\end{equation}
for all $Y \in \mathfrak{X}(M)$. Therefore, the Koszul formula, which defines the Levi-Civita connection, takes the form
\begin{align*}
g(\nabla_X Y, Z) = \left\la(\nabla_X Y)^\flat, Z\right\ra =\frac{1}{2}&\big(  X(g(Y, Z))+Y(g(X, Z))-Z(g(X, Y)) \\
& -g(X,[Y, Z])-g(Y,[X, Z])+g(Z,[X, Y])\big).
\end{align*}
Moreover, the torsion-free and compatibility conditions can be expressed as
\begin{equation}
    \left(\nabla_X Y\right)^\flat-\left(\nabla_Y X\right)^\flat-[X, Y]^\flat=0,
\end{equation}
\begin{equation}
X(g(Y, Z))-\la\left(\nabla_X Y\right)^\flat, Z\ra-\la Y,\left(\nabla_X Z\right)^\flat\ra=0.
\end{equation}
The above observation suggests that the term $(\nabla_X Y)^\flat$ is well-defined in $\D^\prime$, and therefore makes sense for distributional metrics. This motivates the following definition.
\begin{definition}
The \textit{distributional Levi-Civita connection of a distributional metric $g$} is the operator $\nabla^\flat: (X, Y) \in \mathfrak{X}(M) \times \mathfrak{X}(M) \mapsto \nabla^\flat_X Y \in \D^\prime \mathcal{T}^0_1(M)$ defined by
\begin{align*}
\left\la\nabla_X^\flat Y, Z\right\ra:=\frac{1}{2}( & X(g(Y, Z))+Y(g(X, Z))-Z(g(X, Y)) \\
& -g(X,[Y, Z])-g(Y,[X, Z])+g(Z,[X, Y])) \text{ in } \D^\prime(M),
\end{align*}
for all $Z \in \mathfrak{X}(M)$.
\end{definition}
 Note that this is not a distributional connection in the sense previously defined, since it is of type $(0, 1)$ instead of $(1, 0)$. Moreover, the torsion and compatibility conditions hold in $\D^\prime(M)$. Therefore, when $g$ is smooth, it is clear that $\nabla^\flat_X Y = (\nabla_X Y)^\flat$ for all $X, Y \in \mathfrak{X}(M)$, and the torsion-free and compatibility conditions are recovered.

Suppose that $g$ is a metric with regularity $\Lloc^\infty(M) \cap \Wloc^{1,2}(M)$. Then the distributional Levi-Civita connection $\nabla^\flat$ is of class $\Lloc^2$. Suppose that the inverse of the metric is also $\Lloc^\infty(M)$ (or assume that $g$ locally uniformly non-degenerate \cite{LeFloch2007}), then we obtain an $\Lloc^2$-connection $\nabla$ from $\nabla^\flat$ by raising the index via $g$, i.e.,
\begin{equation}
    g\left(\nabla_X Y, Z\right):=\left(\nabla_X^\flat Y\right)(Z)  \text{ in } \D^\prime(M),
\end{equation}
for all $X, Y, Z \in \mathfrak{X}(M)$. This approach defines a $\Lloc^2$-vector field $\nabla_X Y$, hence $\nabla$ is an $\Lloc^2$-connection and the Riemann and Ricci curvature tensors associated with $g$ are well-defined as distributions.

Alternatively, we can define the Riemann curvature tensor of a metric as the distribution $\Riem: \mathfrak{X}(M)^4 \to \D^\prime(M)$ given by
\begin{align*} 
\Riem(W, Z, X, Y):= & X(g(W, \nabla_Y Z))-Y(g(W, \nabla_X Z)) \\ & -g(\nabla_X W, \nabla_Y Z)+g(\nabla_Y W, \nabla_X Z)-g(W, \nabla_{[X, Y]} Z),
\end{align*}
for all $X, Y, Z, W \in \mathfrak{X}(M)$. It is straightforward to verify that 
\begin{equation}
\Riem(W, Z, X, Y ) = g(W, \Riem(X,Y)Z) \text{ in } \D^\prime(M).    
\end{equation}
In order to define the scalar curvature, we must raise an index of the Ricci tensor using $g$. Referring to the Riemann tensor in \Cref{def:RiemmTensorDist}, note that the last three terms lie in $\Lloc^1$, while the first two are distributions. Since raising an index involves multiplication by an $\Lloc^\infty\cap W^{1,2}_\loc$-function, this operation must be understood using the following formula
\begin{equation}\label{eq:ProdRuleDist}
f X(\la \theta, Y \ra) :=  X( f \la \theta, Y \ra) - X(f) \la \theta, Y \ra, \end{equation}
for all $X \in \mathfrak{X}(M)$, $Y \in \D^\prime\mathcal{T}^1_0(M)$, $\theta \in \Omega^1(M)$ and $f \in \Lloc^\infty(M) \cap  W^{1,2}(M)$. Let $E_i$ be a local frame. Then the \textit{distributional scalar curvature},
\begin{equation}
R := \sum_{i,j=1}^n g^{ij} \Ric(E_i, E_j) \in \D^\prime(M),    
\end{equation}
is a well-defined distribution. 

We also introduce the Christoffel symbols associated with $\nabla$. Given a local frame $E_i$, the \textit{Christoffel symbols of an $\Lloc^2$-connection $\nabla$} are the $\Lloc^2$-functions $\Gamma^i_{jk}$ on $M$ defined by
$$
\nabla_{E_k}E_j = \Gamma^i_{jk} E_i,
$$
for all $i,j,k \in \{1,\ldots,n\}$. From the Koszul-formula, we obtain the following local expression for the Christoffel symbols,
\begin{equation}\label{eq:ChristLocal}
\Gamma^{k}_{ij} = \frac{1}{2} g^{kl}\left(\partial_ig_{jl}+ \partial_j g_{il} - \partial_l g_{ij} \right).
\end{equation}
The standard local formula hold in $\D^\prime(M)$:
\begin{align}
\Riem_{i j k}^m & =\partial_j \Gamma_{i k}^m-\partial_k \Gamma_{i j}^m+\Gamma_{j s}^m \Gamma_{i k}^s-\Gamma_{k s}^m \Gamma_{i j}^s, \\
\Ric_{i j} & =\Riem_{i m j}^m, \\
R & =g^{ij} \Ric_{ij}.
\end{align}
Combining the coordinate expression of the Riemann tensor with \eqref{eq:ProdRuleDist} and \eqref{eq:ChristLocal}, the scalar curvature has the following coordinate expression
\begin{align*}
R& :=\partial_m\left(g^{i j} \Gamma_{i j}^m\right)-\left(\partial_m g^{i j}\right) \Gamma_{i j}^m-\partial_j\left(g^{i j} \Gamma_{i m}^m\right)+\left(\partial_j g^{i j}\right) \Gamma_{i m}^m+g^{i j} \Gamma_{i j}^m \Gamma_{k m}^k-g^{i j} \Gamma_{i k}^m \Gamma_{j m}^k\\
& =\partial_m\left(g^{i j} \Gamma_{i j}^m - g^{i m} \Gamma_{i j}^j\right) - \left(\partial_m g^{i j}\right) \Gamma_{i j}^m+\left(\partial_j g^{i j}\right) \Gamma_{i m}^m+g^{i j} \Gamma_{i j}^m \Gamma_{k m}^k-g^{i j} \Gamma_{i k}^m \Gamma_{j m}^k,
\end{align*}
This shows that the local definitions given in \cite[(3.2),(3.3)]{Gra:20} are compatible with the global approach developed here. Furthermore, as justified in \cite{Gra:20}, all calculations can be performed locally. 

For a $C^2$ metric, the coordinate expression of the scalar curvature is usually written as
\begin{equation}\label{eq:CoordScalarC_W1n}
R = \partial_k V^k + F,    
\end{equation}
where the vector field $V$ and the scalar field $F$ are defined by
\begin{equation}\label{eq:V_defSc_W1n}
V^k:= g^{i j} \Gamma_{i j}^k - g^{i k} \Gamma_{ij}^j,    
\end{equation}
\begin{equation}\label{eq:F_defSc_W1n}
F:=  - \left(\partial_m g^{i j}\right) \Gamma_{i j}^m+\left(\partial_j g^{i j}\right) \Gamma_{i m}^m+g^{i j} \Gamma_{i j}^m \Gamma_{k m}^k-g^{i j} \Gamma_{i k}^m \Gamma_{j m}^k.   
\end{equation}
For any compactly supported density $\mu$ in a chart $(U, \psi)$, which can be expressed as $\mu = f dx^1 \wedge\ldots \wedge dx^n$ with $f \in C^\infty_c(U)$, we obtain that 
\begin{align}\label{eq:dist_scalar_W1n}
\la R, \mu\ra
& = \int_{U} -V^k \partial_k f + F f  dx.
\end{align}
Given a partition of unity, the distribution $R$ is uniquely determined by the above expression.

\begin{remark}\label{rmk:DistDef}
In the literature, one can find a distinct definition of distributional scalar curvature proposed by Lee-LeFloch \cite{LeeLeFloch2015}. The difference mainly consists in fixing a smooth background metric $h$, which coincides with the flat metric outside a compact set, and in using test-functions $u \in C^\infty_c(M)$ instead of compactly supported densities $\mu$. Their definition is independent of $h$, and the two definitions agree for $C^2$ metrics. A standard density argument then implies that they also agree for $C^0\cap W^{1,2}_\mathrm{loc}$ metrics.
\end{remark}

The following proposition shows, provided $g$ is sufficiently regular, the distributional scalar curvature can also be extended to less regular test densities. 

\begin{proposition}\label{prop:ExtendingRDist}
Let $(M^n, g)$ be a Riemannian manifold with $ g \in C^0\cap W^{1,p}_{\loc}$, $p\geq 2$. Then, the distributional scalar curvature $R[g]$ can be extended to all compactly supported test densities in $L^{(p/2)^\prime}$ whose derivatives lies in $L^{p^\prime}$, where $p^\prime = p/(p-1)$ and $(p/2)^\prime = p/(p-2)$ for $p>2$, while $(p/2)^\prime=\infty$ if $p=2$.
\end{proposition}
\begin{proof}
    The assumption $g \in C^0 \cap W^{1,p}_\loc$ implies that $V^k \in L^p_{\loc}$ and $Q\in L^{p/2}_{\loc}$, for $p\geq 2$. Fix a smooth Riemannian background metric $h$ with a smooth density $d\mu_h$. Any compactly supported density $d\mu$ can be written as $d\mu = u d\mu_h$, where $u$ is a compactly supported function with the same regularity as the density $d\mu$. Without loss of generality, assume that $d\mu$ is compactly supported in a chart $(U, \psi)$. From \eqref{eq:dist_scalar_W1n}, we obtain that $\la R[g], u \,d\mu_h \ra$ is well-defined provided that the following is integrable
    \begin{equation}
    V^k \partial_k u \sqrt{|h|} + V^k  u \partial_k \left( \sqrt{|h|}\right) + Q u \sqrt{|h|}.
    \end{equation}
     The first term is integrable since $\frac{1}{p}+\frac{1}{p^\prime} = 1$. For the remaining two terms, if $p=2$, then $(p/2)^\prime = \infty$, so integrability follows from $u \in L^\infty$, $V \in L^2_\loc$, and $Q \in L^{1}_\loc$. If $p>2$, we have that $(p/2)^\prime = p/(p-2)>p/(p-1)=p^\prime$, then $u \in  L^{(p/2)^\prime}_\loc \subset L^{p^\prime}_\loc$. Finally, the integrability is obtained by noticing that $\frac{1}{p}+\frac{1}{p^\prime} = 1$ and $\frac{2}{p}+\frac{1}{(p/2)^\prime} = 1$ as long as $p>2$.

\end{proof}

\subsection{Regularisation}

In what follows, we introduce the regularisation of non-smooth tensor fields via chartwise convolution of their components with a mollifier. Let $(U_i, \psi_i)_{i \in \mathbb{N}}$ be a countable and locally finite family of relatively compact coordinate charts $\psi:U_i \to B_1(0)\subset \RR^n$ covering $M$, and let $(\eta_i)_{i \in \mathbb{N}}$ be a subordinate partition of unity with $\operatorname{supp}(\eta_i) \subseteq U_i$. Moreover, let $(\chi_i)_{i \in \mathbb{N}}$ be a family of cut-off functions $\chi_i \in C^\infty_c(U_i)$ such that $\chi_i \equiv 1$ on a neighbourhood of $\operatorname{supp}(\eta_i)$.  Let $\rho:\RR^n \to \RR$ be a smooth, positive mollifier supported in the unit ball with unit integral, and define $\rho_\eps(x):= \eps^{-n} \rho(x/\eps)$, where $\eps \in (0,1]$. Denoting by $f_*$ and $f^*$ the push-forward and pull-back of distributions under a diffeomorphism $f$, respectively, we define the regularisation of any $\mathcal{T} \in \D^\prime \mathcal{T}^r_s(M)$ as a smooth $(r, s)$-tensor field $\mathcal{T}\star_M \rho_\eps$ via the expression
\begin{equation}\label{eq:reg_tensor}
    \mathcal{T} \star_M \rho_\eps (x) := \sum_{i \in \mathbb{N}} \chi_i (x) (\psi_i)^* \left[ \big((\psi_i)_*(\eta_i \mathcal{T})\big) * \rho_\eps\right] (x),
\end{equation}
where $(\psi_i)_*(\eta_i\mathcal{T})$ is understood as a compactly supported distributional tensor field on $\RR^n$, and convolution with $\rho_\eps$ is component-wise, resulting in a smooth field on $\RR^n$. The cut-off functions $\chi_i$ are introduced in order to guarantee that the map $(\eps, x) \mapsto \mathcal{T} \star_M \rho_\eps$ is smooth on $(0, 1] \times M$. Note that for compact sets and sufficiently small $\eps$, all but finitely many terms vanish, and  the corresponding cut-off functions satisfy $\chi_i \equiv 1$. In such cases, the cut-off functions can be omitted. This regularisation procedure provides the following convergence properties, analogous to those of classical convolution smoothing on $\RR^n$.

\begin{proposition}[Convergence properties of $\mathcal{T} \star_M \rho_\eps$ \cite{Gra:20}]\label{prop:CongTeps}
We have:
\begin{enumerate}
    \item If $\mathcal{T} \in \D^{\prime} \mathcal{T}_s^r(M)$, then
\begin{equation}
\left\la\mathcal{T} \star_M \rho_{\eps}, \phi\right\ra \to\la\mathcal{T}, \phi\ra \quad \forall \phi \in \Gamma_c\left(M, T_r^s M \otimes \Vol(M)\right) .    
\end{equation}
\item If $\mathcal{T}$ is $C^k$, then $\mathcal{T} \star_M \rho_{\eps} \to \mathcal{T}$ in $C_{\mathrm{\mathrm{loc}}}^k$. Further, if $k \geq 1$, then for any compact $K \subseteq M$ there exists a constant $c_K>0$ and $\eps_0(K)>0$ such that
\begin{equation}
\left\|\mathcal{T}-\mathcal{T} \star_M \rho_{\eps}\right\|_{\infty, K} \leq c_K \eps,    
\end{equation}
for all $\eps<\eps_0$.
\item If $\mathcal{T} \in \Wloc^{k,p} T^r_s(M)$, $1 \leq p<\infty$ and $k\geq1$, then $\mathcal{T} \star_M \rho_{\eps} \to \mathcal{T}$ in $\Wloc^{k, p}$.
\end{enumerate}
  
\end{proposition}

Let $g$ be a $C^0$-Riemannian metric, i.e., a symmetric, non-degenerate $(0, 2)$-tensor field. We now introduce $g \star_M \rho_\eps$ to obtain a net of smooth Riemannian metrics that converges to $g$ in $C^0_{\mathrm{loc}}$. Using the general formula for the inverse of $g$, and \Cref{prop:CongTeps}, and given that $g$ is $C^0$, we have $(g \star_M \rho_\eps)^{-1} \to g^{-1}$ in $\Cloc^0$ as well. As discussed in the previous section, if the metric $g$ additionally belongs to $\Wloc^{1,2}(M)$, then $g$ admits an $\Lloc^2$ Levi-Civita connection $\nabla$, and the Riemann tensor, Ricci tensor, and scalar curvature are properly defined. From \cite[Theorem 4.6]{LeFloch2007} and the convergence of the inverse metrics, we obtain a so-called \textit{stability result}, i.e., the convergence of curvature quantities under regularisation.

\begin{proposition}[\cite{LeFloch2007}]
Let $(M, g)$ be a Riemannian manifold with $g \in C^0\cap \Wloc^{1,2}$. Then 
\begin{equation}
R[g \star_M \rho_\eps] \to R[g]  \text{ in } \D^\prime(M).    
\end{equation}
In particular, the Levi-Civita connection of $g\,\star_M \rho_\eps$, as well as $\Riem[g\,\star_M \rho_\eps]$ and $\Ric[g\,\star_M \rho_\eps]$, converge distributionally in an analogous manner.
\end{proposition}

In order to discuss bounds on scalar curvature, we first need to define a notion of positivity for distributions. Distributions on $M$ admit a natural definition of positivity, given as follows.
\begin{definition}[Positivity of Distributions]\label{def:PosDistributions}
    Let $u \in \D^\prime(M)$. We say that $u \geq 0$ (resp. $u > 0$) if holds that $\la u, \mu \ra\geq 0$ (resp. $\la u, \mu \ra > 0$) for all compactly supported nonnegative (resp. positive) volume densities $\mu$ on $M$. Moreover, for any $u, v \in \D^\prime(M)$, we define $u \geq v$ (resp. $u<v$) if the distribution $u-v$ is nonnegative (resp. positive) in the above sense.
\end{definition}
Clearly, for a distribution $u \in \D^\prime(M)$ we have 
\begin{equation}
u\geq0 \implies u \star_M \rho_\eps\geq 0.
\end{equation}

In the following, a convenient family of regularised metrics $g_\eps$ is introduced, which coincides with the original metric outside a compact. As a consequence, scalar curvature remains unchanged outside a compact set, and the ADM mass is also preserved.

\begin{lemma}[\cite{grant2014positive}] \label{lemma:FamilyOfMetrics}
    Let $g \in C^0 \cap \Wloc^{k,p}$, with $k, p\geq1$, and let $K\subset M$ be a compact set such that $g$ is smooth on $M \setminus K$. Then, for all $\eps>0$, there exists a smooth Riemannian metric $g_\eps$ and a compact set $K_0 \subset M$ containing $K$ with the following properties:
    \begin{enumerate}
        \item $g_\eps$ converges to $g$ in $\Wloc^{k,p}(M)$ and locally uniformly as $\eps \to 0$;
        \item $g_\eps \equiv g$ on $M\backslash K_0$.
    \end{enumerate}
\end{lemma}
\begin{proof}
    For completeness, we present the result of \cite{grant2014positive} in our setting. The construction of the regularised metric closely follows the procedure for smoothing tensor fields discussed earlier, but the choice of open cover of $M$ requires particular care. Let $K \subset M$ be a compact set such that $g$ is smooth on $M \setminus K$, and denote its complement by $\mathcal{C}:= M \setminus K$. We cover $K$ by a finite collection of coordinate charts $\psi_i: U_i \to B_1(0)\subset \RR^n$, where $i=1,\ldots, m$. Consequently, we obtain an open cover of $M$ given by $M = \mathcal{C} \cup \left(\cup^m_{i=1} U_i\right)$. 
    
    Let $\eta_i$, $i=1,\ldots, m$, and $\eta_{\mathcal{C}}$ be a partition of unity subordinate to the open cover $\mathcal{C} \cup \left(\cup^m_{i=1} U_i\right)$ of $M$ with the property that $\supp(\eta_i \circ \psi_i^{-1}) \subset B_1(0)$ and that $\eta_\mathcal{C}$ has support bounded away from $\partial \mathcal{C}$. Let $\rho_\eps$ be a standard mollifier, and let $\{\chi_1, \ldots, \chi_m\}$ be a family of cut-off functions $\chi_i \in C^\infty_c(U_i)$, with $\chi_i \equiv 1$ on a neighbourhood of $\operatorname{supp}(\eta_i)$. We define the smooth metric $g_\eps$ via the expression
    \begin{equation}
    g_\eps := \eta_\mathcal{C} g + \sum_{i=1}^m  \chi_i (\psi_i)^* \left[ \big((\psi_i)_*(\eta_i g)\big) * \rho_\eps\right].
    \end{equation}
    By \Cref{prop:CongTeps}, we have that $g_\eps \to g$ in $\Wloc^{k,p}$ and locally uniformly. Note that the metrics $g$ and $g_\eps$ coincide outside a compact set $K_0$ containing $K$, defined as the union of the supports of all the $\chi_i$.
\end{proof}

From now on, we denote by $g_\eps$ the family of smooth metric obtained by \Cref{lemma:FamilyOfMetrics}, and in the following, we show some useful properties of $g_\eps$ for the subsequent sections.
\begin{lemma}[\cite{grant2014positive}]
    Let $g$ be a continuous metric on $M$. Then, there exists $\rho(\eps)\geq 1$ such that
    \begin{equation}
        \frac{1}{\rho(\eps)} g_\eps \leq g \leq \rho(\eps) g_\eps
    \end{equation}
    as bilinear forms on $M$, with $\rho(\eps) \to 1$ as $\eps \to 0$.
\label{lemma:bilinearmetric}
\end{lemma}
\begin{proof}
    For completeness, we restate the proof of \cite{grant2014positive} here. By construction, $g_\eps$ converges uniformly to $g$ on compact subsets. Considering the compact set $K_0$, it follows that there exists $\rho(\eps)$ such that the inequality holds on $K_0$. Since $g_\eps$ coincides with $g$ on $M \setminus K_0$, it follows that the inequality holds globally on $M$ for each $\eps>0$. The uniform convergence of the metrics on $K$ imply that $\rho(\eps) \to 1$ as $\eps \to 0$.
\end{proof}
Schoen and Yau \cite[Lemma 3.1]{PMTSchoenYau} showed that for any smooth asymptotically flat metric $h$, there exists a constant $C>0$ such that
\begin{equation}\label{eq:SobolevConstant0}
\|\varphi\|_{L^{n^*}(M, h)}^2 \leq C \|\nabla^h \varphi \|^2_{L^2(M, h)},
\end{equation}
where $n^*:= \frac{2n}{n-2}$, and $\varphi$ is any function with compact support on $M$.

The smallest constant satisfying inequality \eqref{eq:SobolevConstant0} is called the \textit{Sobolev constant} $C(h)$ of the metric $h$. Following the argument proposed in \cite[Lemma 3.1]{PMTSchoenYau}, we conclude that the inequality remains valid even for asymptotically flat $C^0$ metrics, i.e., for any function $\varphi$ with compact support on $M$, we have that
\begin{equation}\label{eq:SobolevConstant}
\|\varphi\|_{L^{n^*}(M, g)}^2 \leq C(g) \|\nabla^g \varphi \|^2_{L^2(M, g)},    
\end{equation}
where $C(g)>0$ is the Sobolev constant of the metric $g$.
\begin{proposition}[\cite{grant2014positive}]\label{prop:SobolevConstant}
Let $g$ be a complete asymptotically flat $C^0$ metric on $M^n$. Then the Sobolev constants $C(g)$ and $C(g_{\eps})$ of the metrics $g$ and $g_\eps$, respectively, satisfy the inequality
\begin{equation}
\frac{1}{\rho(\eps)^n} C(g_{\eps}) \leq C(g) \leq \rho(\eps)^n C(g_\eps),    
\end{equation}
where $\rho$ is the function from \Cref{lemma:bilinearmetric}.
\end{proposition}
\begin{proof}
Once again, for completeness, we present the proof from \cite{grant2014positive} here. Note that for any test function $\varphi$, by \Cref{lemma:bilinearmetric} and \eqref{eq:SobolevConstant}, we have
\begin{align*}
\|\varphi\|_{L^{n^*}(M, g_\eps)}^{2} & =  \left(\int_M |\varphi|^{n^*} \,\dgeps \right)^{2/n^*} \\
&\leq  \left(\int_M |\varphi|^{n^*} \rho(\eps)^{\frac{n}{2}}\,\dg \right)^{2/n^*} \\
& \leq \rho(\eps)^{\frac{n-2}{2}} C(g) \left(\int_M  |\nabla^g \varphi|^{2} \,\dg \right) \\
& \leq \rho(\eps)^{\frac{n-2}{2}} C(g) \left(\int_M  \rho(\eps)^{1 + n/2} |\nabla^{g_\eps} \varphi|^{2} \,\dgeps \right) \\
& \leq \rho(\eps)^{n} C(g) \|\nabla^{g_\eps} \varphi\|_{L^2(M, g_\eps)}^2,
\end{align*}
which implies that $C(g_\eps) \leq \rho(\eps)^{n} C(g)$. By reversing the argument, we establish the desired result.
\end{proof}

\section{Approximations and scalar curvature}\label{section:ScalarCurvature}

\subsection{Friedrichs' lemma}

Before delving into the behaviour of the scalar curvature for regularised metrics, we establish an adaptation of Friedrichs' Lemma, extending it to Sobolev spaces (see \Cref{prop:FriedrichsWp}). The proof presented here follows \cite[Lemma 4.8]{Gra:20} and \cite[Lemma 17.1.5]{Hormander2007-ch}, and, although it is essentially the same argument as in \cite[Lemma 3.3]{Calisti2025}, it does not require the Lipschitz property and is presented differently\footnote{This result was obtained in parallel with \cite{Calisti2025}, however, during the preparation of this draft, it was found that the Sobolev version of Friedrichs' lemma was also established in \cite{mondino2025}.}.

In this section, we work on $\RR^n$, which is sufficient for extending the results to manifolds, as discussed at the end of \Cref{subsection:DistCurvature}. We denote by $\partial_j$ the partial derivative $\frac{\partial}{\partial x^j}$, where $j=1,\ldots,n$, and by $\rho_\eps$ the standard mollifier, i.e., $\rho_\eps(x) =\rho(\frac{x}{\eps})/{\eps^n}$, where $\eps \in (0,1]$. For any $p \in [1, +\infty)$, we define the \textit{conjugate exponent} $p^\prime$ of $p$ by $\frac{1}{p}+\frac{1}{p^\prime}=1$, with the usual conventions for $p=1$ and $p=+\infty$.

In preparation for what follows, and inspired by \cite[Lemma 4.7]{Gra:20}, we show that if $f$ is merely integrable in $L^p$, $p \in [1, \infty)$, i.e., without weak derivatives, then $\partial_j(f * \rho_\eps)$ generally diverges as $\eps \to 0$, though at a controlled rate. 

\begin{lemma}{\label{lemma:AuxFriedrichswp}}
Let $f \in \Lloc^{p}(\RR^n)$, with $p \in [1,\infty)$, and let $\rho_\eps$ be a standard mollifier. Then, for any compact set $K \subset \RR^n$, $\eps \|\partial_j(f * \rho_\eps)\|_{L^{p}(K)} \to 0$ as $\eps \to 0$. 
\end{lemma}
\begin{proof}
By \cite[Lemma 4.7]{Gra:20}, if $f$ is continuous, then $\eps \|\partial_j(f * \rho_\eps)\|_{K,\infty} \to 0$ as $\eps \to 0$, for any compact set $K \subset \RR^n$. We extend this result for $p$-integrable functions by a density argument. Let $K\subset \RR^n$ be a compact set. Since $p<\infty$, there exists a sequence $(f_n)_{n \in \mathbb{N}}$ of functions in $C^0_c(\RR^n)$ such that $f_n \to f$ in $\Lloc^p(\RR^n)$. For each $n$, it follows $\eps \|\partial_j(f_n * \rho_\eps)\|_{K, \infty} \to 0$ as $\eps \to 0$. 

For any $f \in \Lloc^p$, the following equality holds
\begin{equation}
    \partial_j(f* \rho_\eps)(x) = \int_{\RR^n} f(y) \partial_j^x\rho_\eps(x-y)dy=-\int_{B_\eps} f(x-y) \partial_j\rho_\eps(y)dy,
\end{equation}
and we have the estimate
\begin{align}
    \int_K \left| \partial_j(f* \rho_\eps)(x)\right|^p dx 
    & \leq  \int_K \left| \int_{B_\eps} \left|f(x-y) \partial_j \rho_\eps(y)\right| dy \right|^p dx.
\end{align}
Recall that $p^\prime$ is the conjugate exponent of $p$. Then, by H\"{o}lder's inequality, we obtain
\begin{align}
\int_{B_\eps} |f(x-y) \partial_j \rho_\eps(y)| dy & = \int_{B_\eps} \left(|f(x-y)| |\partial_j \rho_\eps(y)|^{1/p}\right)  |\partial_j \rho_\eps(y)|^{1/p^\prime} dy \\
& \leq \left(\int_{B_\eps} \left(|f(x-y)|^p |\partial_j \rho_\eps(y)|\right) dy\right)^{1/p} 
 \left(\int_{B_\eps} |\partial_j \rho_\eps(y)| dy\right)^{1/p^\prime}, 
\end{align}
and
\begin{align}\label{eq:JensenIneq}
    \left| \int_{B_\eps} f(x-y) \partial_j \rho_\eps(y)  dy\right|^p & \leq \left(\int_{B_\eps} |f(x-y)|^p |\partial_j \rho_\eps(y)| dy\right) 
 \|\partial_j\rho_\eps\|_{L^1(B_\eps)}^{p-1}, 
\end{align}
where this last inequality is a form of Jensen's inequality. 

Fix a compact set $K \subset \RR^n$. Let $\eps_1>0$, then there exists $N>0$ such that $\|f-f_n\|_{L^p(K_1)}\leq \eps_1$ for all $n\geq N$, where $K_{1}:=\{x \in \RR^n\,|\, d(x, K)\leq 1\}$. Given any $n\geq N$ and $\eps<1$, substituting $f$ by $f-f_n$ in Jensen's inequality yields
\begin{align*}
    \eps^p \int_K \left| \partial_j((f - f_n)* \rho_\eps)(x)\right|^p dx     &  =  \eps^p \int_K \left| \int_{B_\eps} (f-f_n)(x-y) \partial_j \rho_\eps(y) dy\right|^p dx   \\
    & \leq  \eps^p  \|\partial_j \rho_\eps\|_{L^1(B_\eps)}^{p-1} \int_K  \int_{B_\eps} |(f-f_n)(x-y)|^p |\partial_j \rho_\eps(y)|  dy  dx\\
    & \leq  \eps^p \|\partial_j \rho_\eps\|_{L^1(B_\eps)}^{p-1} \int_{B_\eps} |\partial_j\rho_\eps(y)| \int_{K_1}  |(f-f_n)(x)|^p   dxdy \\
    & \leq \eps_1^p  \eps^p \|\partial_j \rho_\eps\|_{L^1(B_\eps)}^{p-1} \int_{B_\eps} |\partial_j\rho_\eps(y)| dy \\
    & \leq \eps_1^p \eps^p \|\partial_j \rho_\eps\|_{L^1(B_\eps)}^{p} = \eps_1^p \| \eps \partial_j \rho_\eps\|_{L^1(B_\eps)}^{p}  = \eps_1^p \|\partial_j \rho\|_{L^1(B_1)}^{p},
\end{align*}
where the last equality follows from $\int_{B_{\eps}}\left|\eps \partial_j \rho_{\eps}(y)\right| dy=\int_{B_1}\left|\partial_j \rho(y)\right| d y$. Using the above inequality, we then obtain
\begin{align*}
\eps \|\partial_j(f * \rho_\eps)\|_{L^p(K)}&  \leq  \eps \|\partial_j((f-f_n) * \rho_\eps)\|_{L^p(K)} + \eps \|\partial_j(f_n * \rho_\eps)\|_{L^p(K)} \\
& \leq    \eps_1 \|\partial_j \rho \|_{L^1(B_1)} + \eps \|\partial_j(f_n * \rho_\eps)\|_{K, \infty} \cdot \Vol(K)^\frac{1}{p}.
\end{align*}
Taking the limit superior as $\eps \to 0$, we find
\begin{equation}
\limsup_{\eps \to 0} \eps \|\partial_j(f * \rho_\eps)\|_{L^p(K)} \leq  \eps_1 \|\partial_j \rho \|_{L^1(B_1)}.    
\end{equation}
Since $\eps_1$ is arbitrary, we conclude $\eps \|\partial_j(f * \rho_\eps)\|_{L^p(K)} \to 0$ as $\eps \to 0$.
\end{proof}

The result indicates that controlling convergence rates can play a key role in managing derivatives. We now present the following estimate, which extends \cite[Proposition 3.5, item 2]{Gra:20} from differentiable functions to Sobolev functions.

\begin{lemma}\label{lemma:RateConvSobolev}
Let $a \in \Wloc^{1,p}\left(\RR^n\right)$ and $f \in \Lloc^q\left(\RR^n\right)$, with $p, q \in [1, +\infty)$, $q\geq p^\prime$. Then, for any compact $K \subset \RR^n$, there exists a constant $C>0$ such that
\begin{equation}
\|(a * \rho_\eps) (f* \rho_\eps) - (af)* \rho_\eps \|_{L^r(K)} \leq C \eps,    
\end{equation}
where $r$ is defined by $1/r=1/p+1/q$.
\end{lemma}
\begin{proof}
Fix a compact set $K \subset \RR^n$ and let $x \in K$. Consider the following identity,
\begin{align}
   (a * \rho_\eps) (f* \rho_\eps) - (af)* \rho_\eps & = \big[(a * \rho_\eps - a) (f* \rho_\eps)\big] +  \big[a(f* \rho_\eps)-(af)* \rho_\eps\big].
\end{align}
For any $\eps<1$, by H\"{o}lder's inequality and the fact that $\|f* \rho_\eps\|_{L^q(K)} \leq \|f\|_{L^q(K_1)}$, where $K_{1}:=\{x \in \RR^n\,|\, d(x, K)\leq 1\}$, the following inequality holds
$$
\| (a * \rho_\eps) (f* \rho_\eps) - (af)* \rho_\eps\|_{L^r(K)} \leq \|a * \rho_\eps - a\|_{L^p(K)} \|f\|_{L^q(K_1)} +  \|a(f* \rho_\eps) - (af)* \rho_\eps\|_{L^r(K)},
$$
For the second term on the right-hand side, by the definition of convolution, we have
\begin{align*}
 \big(a (f* \rho_\eps) - (af) * \rho_\eps\big) (x) & = \int_{B_\eps} \big(a(x) - a(x-y)\big) f(x-y) \rho_\eps(y) dy\\
& = \eps \int_{B_1} \frac{\big(a(x) - a(x-\eps z)\big)}{\eps} f(x-\eps z)\rho (z) dz.
\end{align*}
Suppose, for a moment, that $a$ is a test function. Then,
\begin{equation}
\big(a (f* \rho_\eps) - (af) * \rho_\eps\big) (x) = \eps \int_{B_1} \int_0^1 \big[Da{(x-t\eps z)} \cdot z\big] f(x-\eps z)\rho (z) dtdz, 
\end{equation}
Integrating over $K$, applying Jensen's inequality~\eqref{eq:JensenIneq} and H\"{o}lder's inequality, we have
\begin{align*}
\int_K \left| a (f* \rho_\eps) - (af * \rho_\eps)\big) (x)\right|^r dx 
& = \eps^r \int_K \left|\int_{B_1} \int_0^1 [Da{(x-t\eps z)} \cdot z ]  f(x-\eps z) \rho(z)  dt dz \right|^rdx\\
& \leq \eps^r  \int_K \int_{B_1} \int_0^1 \left| [Da{(x-t\eps z)} \cdot z ]  f(x-\eps z)\right|^r \rho(z)  dt dz dx\\
& \leq \eps^r \|\rho\|_{B_1, \infty} \int_K \int_{B_1} \int_0^1 \left| [Da{(x-t\eps z)} \cdot z ]  f(x-\eps z)\right|^r  dt dz dx\\
& \leq \eps^r \|\rho\|_{B_1, \infty} \left(\int_K \int_{B_1} \int_0^1 \big| |Da|{(x-t\eps z)}\big|^p  dt dz dx\right)^{r/p}\\
&\quad  \times  \left(\int_K \int_{B_1} \int_0^1  \big|f(x-\eps z)\big|^q  dt dz dx\right)^{r/q}.
\end{align*}
Note that, for any $t \in [0,1]$ and $z \in B_1$, the maps $x \mapsto x-t\eps z$ and $x \mapsto x-\eps z$ maps $K$ onto some bounded open subset of $\RR^n$. Thus, by change of variables, we conclude that
\begin{equation}
\| a (f* \rho_\eps) - (af * \rho_\eps)\|_{L^r(K)} \leq c_1 \eps, \end{equation}
where $c_1>0$ is some positive constant. Since smooth functions are dense in $\Wloc^{1,p}(\RR^n)$, for $p<\infty$, the inequality holds for $a \in \Wloc^{1,p}(\RR^n)$. Taking the previous inequality with $f \equiv 1$ and $r\equiv p$, we have
\begin{equation}\label{eq:ConvRateMollinLp}
\|a* \rho_\eps - a \|_{L^p(K)} \leq c_2 \eps,    
\end{equation}
for some constant $c_2>0$.
Therefore, we have obtained that 
\begin{equation}
\|(a * \rho_\eps) (f* \rho_\eps) - (af)* \rho_\eps\|_{L^r(K)} \leq C \eps,    
\end{equation}
for some constant $C>0$.
\end{proof}

Now we present our first version of Friedrichs' Lemma. The following approach is inspired by the Friedrichs commutator estimates found in \cite[Theorem III.2.10]{Boyer2013}.

\begin{proposition}\label{prop:FriedrichsWp}
For $a \in \Wloc^{1,p}\left(\RR^n\right)$ and $f \in \Lloc^q\left(\RR^n\right)$, with $p, q \in [1, +\infty)$, $q\geq p^\prime$, we have that 
\begin{equation}
    \|\left(a * \rho_{\eps}\right)\left(f * \rho_{\eps}\right)-(a f) * \rho_{\eps}\| _{W^{1,r}(K)} \longrightarrow 0 \text{ as } \eps \to 0,
\end{equation}
where $r$ is defined by $1/r=1/p+1/q$, for any compact $K \subset \RR^n$.
\end{proposition}
\begin{proof}
In this proof, given a function $f$, we denote $f_\eps := f * \rho_\eps$ and $\partial_j f_\eps := \partial_j(f *\rho_\eps)$. Fix a compact set $K \subset \RR^n$ and define $K_{1}:=\{x \in \RR^n\,|\, d(x, K)\leq 1\}$. The convergence in $L^r_\mathrm{loc}(\RR^n)$ follows directly from \Cref{lemma:RateConvSobolev}.

Now, we proceed to demonstrate the convergence of first-order derivatives, aiming to establish convergence in $\Wloc^{1,r}(\RR^n)$. For any function $f$, it holds that $\partial_j f_\eps := \partial_j(f* \rho_\eps) =  f * \partial_j \rho_\eps$, then
\begin{align*}
\partial_j (a_\eps f_\eps - (af)_\eps)  & = \partial_j ((a_\eps -a ) f_\eps + af_\eps - (af)_\eps)\\
& = \partial_j ((a_\eps -a ) f_\eps) + \partial_j(af_\eps) - \partial_j(af)_{\eps}.
\end{align*}
The first term on the right-hand side can be controlled by the following inequalities
\begin{align*}
\|\partial_j \left((a_\eps -a ) f_\eps\right)\|_{L^r(K)} & \leq \|\partial_j (a_\eps -a ) f_\eps)\|_{L^r(K)} + \| (a_\eps -a) \partial_j(f_\eps)\|_{L^r(K)} \\
& \leq \|\partial_j (a_\eps -a )\|_{L^p(K)}  \|f_\eps\|_{L^q(K)} + \|a_\eps -a\|_{L^p(K)} \| \partial_j(f_\eps)\|_{L^q(K)}\\
& \leq \|a_\eps -a\|_{W^{1,p}(K)}  \|f\|_{L^q(K_1)} + c_2 \eps \| \partial_j(f_\eps)\|_{L^q(K)},
\end{align*}
where $c_2$ is the constant obtained from \eqref{eq:ConvRateMollinLp}. Since $a_\eps \to a$ in $\Wloc^{1,p}(\RR^n)$ as $\eps \to 0$ and applying \Cref{lemma:AuxFriedrichswp}, we obtain 
\begin{equation}\label{eq:Partial_j_Fried}
    \|\partial_j ((a_\eps -a ) f_\eps)\|_{L^r(K)} \longrightarrow 0 \text{ as } \eps \to 0, \text{ for any compact $K \subset \RR^n$.} 
\end{equation}
Then, it remains to deal with the last two terms at the right side. Temporarily, suppose that $a$ is a test function, i.e., smooth with compact support. Let $x \in K$ and define $A_\eps := \partial_j(af_\eps) - \partial_j(af)_\eps$, then
\begin{align*}
    A_\eps(x) & = \frac{\partial}{\partial x^j} \left(a(x) \int_{B_\eps} f(y) \rho_\eps(x-y)dy\right) - \int_{B_\eps} a(y) f(y) \frac{\partial}{\partial x^j} \big(\rho_\eps(x-y)\big) dy \\
    & = \frac{\partial a}{\partial x^j}(x) f_\eps(x) + \int_{B_\eps} \big(a(x)-a(y)\big) f(y) \frac{\partial}{\partial x^j}\big( \rho_\eps(x-y)\big) dy.
\end{align*}
By using that $\partial_j^x\big( \rho_\eps(x-y)\big) = - \partial_j^y\big( \rho_\eps(x-y)\big)$, we have,
\begin{align*}  
     A_\eps(x)  & =\frac{\partial a}{\partial x^j}(x) f_\eps(x) + \int_{B_\eps} \big(a(y)-a(x)\big) f(y) \frac{\partial}{\partial y^j}\big( \rho_\eps(x-y)\big) dy\\
     & = \frac{\partial a}{\partial x^j}(x)  f_\eps(x) + \int_{B_1} \frac{a(x-\eps z)-a(x)}{\eps} f(x-\eps z) \frac{\partial \rho}{\partial z^j}(z) dz\\
    & = \frac{\partial a}{\partial x^j}(x)  f_\eps(x) - \int_{B_1}\int_0^1 \big[ Da{(x-t\eps z)}\cdot z \big] f(x-\eps z) \frac{\partial \rho}{\partial z^j}(z) dtdz.
\end{align*}
By Jensen's inequality~\eqref{eq:JensenIneq} and H\"{o}lder's inequality, it follows that
\begin{align*}
\|A_\eps\|_{L^r(K)} & \leq \|Da\|_{L^p(K)} \|f_\eps\|_{L^q(K)}\\
& +\left( \|\partial_j\rho\|_{L^1(B_1)}^{r-1}  \int_{K} \int_{B_1} \int_0^1 \big| |Da|{(x-t\eps y)} f(x-\eps y)\big|^r |\partial_j\rho(y)| dtdydx\right)^\frac{1}{r}\\
& \leq \|Da\|_{L^p(K)} \|f_\eps\|_{L^q(K)} +\left( \|\partial_j\rho\|_{L^1(B_1)}^{r-1} \int_{B_1} |\partial_j\rho(y)| dy\right)^\frac{1}{r}  \|Da\|_{L^p(K_1)} \|f\|_{L^q(K_1)}\\
& \leq \|Da\|_{L^p(K)} \|f\|_{L^q(K_1)} + \|\partial_j \rho\|_{L^1(B_1)} \|Da\|_{L^p(K_1)} \|f\|_{L^q(K_1)},
\end{align*}
Thus, we conclude that
\begin{equation}
\|A_\eps\|_{L^r(K)} \leq C \|Da\|_{L^{p}(K_1)}\|f\|_{L^q(K_1)},    
\end{equation}
for some positive constant $C$. This estimate holds if $a$ is smooth, and thus is valid by approximation for arbitrary $a \in \Wloc^{1,p}(\RR^n)$.

Let us denote by $\mathcal{A}_\eps$ the bilinear map which, to any $(a, f) \in \Wloc^{1,p}(\RR^n) \times \Lloc^q(\RR^n)$, associates the term $A_\eps \in \Lloc^r(\RR^n)$. Then, we have proved that
\begin{equation}
\|\mathcal{A}_\eps(a, f)\|_{L^r(K)} \leq C \|a\|_{W^{1, p}(K_1)}\|f\|_{L^q(K_1)},    
\end{equation}
for any compact $K$.
Since $A_\eps = \partial_j(af_\eps) - \partial_j(af)_{\eps} \to 0$ as $\eps \to 0$ if $a$ and $f$ are smooth, then it holds that
\begin{equation}
    \|\mathcal{A}_\eps(a, f)\|_{L^r(K)} \longrightarrow 0 \text{ as } \eps \to 0.
\end{equation} 
For $p, q< \infty$, smooth functions known to be dense in $\Wloc^{1,p}(\RR^n)$ and $\Lloc^q(\RR^n)$. For any fixed smooth function $f$, the linear operator $\mathcal{A}_\eps(\cdot, f): W^{1, p}(K_1) \to L^r(K)$ corresponds to a family of uniformly bounded linear operators that converges to zero on a dense subset as $\eps \to 0$. Applying this argument for each entry, we conclude that
\begin{equation} \label{eq:CongA_eps}
\|\mathcal{A}_\eps(a, f)\|_{L^r(K)} \longrightarrow 0 \text{ as } \eps \to 0,
\end{equation}
for any $(a, f) \in \Wloc^{1,p}(\RR^n) \times \Lloc^q(\RR^n)$ and for any compact $K$.

To sum up, let any pair $(a, f) \in \Wloc^{1,p}(\RR^n) \times \Lloc^q(\RR^n)$, then we have
\begin{align*}
\|\partial_j (a_\eps f_\eps - (af)_\eps)\|_{L^r(K)}& \leq     \|\partial_j ((a_\eps -a ) f_\eps)\|_{L^r(K)} + \|\partial_j(af_\eps) - \partial_j(af)_\eps\|_{L^r(K)}\\
& \leq \|a_\eps -a\|_{W^{1,p}}  \|f\|_{L^q(K_1)} + c_2 \eps \| \partial_j(f_\eps)\|_{L^q(K)} + \| \mathcal{A}_\eps(a, f)\|_{L^r(K)},
\end{align*}
which converges to zero as $\eps \to 0$ by \eqref{eq:CongA_eps} and \eqref{eq:Partial_j_Fried}.
\end{proof}

To handle the scalar curvature, the preceding result requires an adaptation. We replace $a * \rho_\eps$ with a more general net $a_\eps \to a$ as $\eps \to 0$, which converges sufficiently fast to $a$. This adjustment is necessary because, when examining the scalar curvature of the regularised metric in coordinates, the terms depend on the inverse of the metric $g\,\star_M \rho_\eps$ rather than on $g^{-1} \star_M \rho_\eps$. 

\begin{corollary}
\label{cor:ConvFriedPlusApproximations}
    Let $p,q \in [1,+\infty)$, $q \geq p^\prime$. Let $a, a_\eps \in C^0(\RR^n)\cap \Wloc^{1,p}\left(\RR^n\right)$ and assume that $a_\eps \to a$ in $\Wloc^{1,p}(\RR^n)$ and in $\Cloc^0\left(\RR^n\right)$ as $\eps \to 0$. Additionally, suppose that there exists $c_K$ such that $\| a - a_\eps\|_{L^p(K)} \leq c_K \eps$ for any compact $K$. Let $f \in \Lloc^q\left(\RR^n\right)$, then 
    \begin{equation}
        a_\eps  (f * \rho_\eps) - (a f)* \rho_\eps \longrightarrow 0 \text{ in } \Wloc^{1,r}(\RR^n) \text{ as } \eps \to 0,
    \end{equation}
    where $r$ is defined by $1/r=1/p+1/q$.
\end{corollary}
\begin{proof}
    Consider the following identity
\begin{align}
      a_\eps (f*\rho_\eps) - (af)* \rho_\eps& = \big[a_\eps - (a*\rho_\eps) \big] (f*\rho_\eps) 
  + \big[ (a*\rho_\eps) (f *\rho_\eps) - (af)* \rho_\eps \big].
\end{align}
The convergence in $\Wloc^{1,r}(\RR^n)$ of the second term is an immediate consequence from \Cref{prop:FriedrichsWp}, so it remains to investigate the first term. Let $K$ be a compact subset of $\RR^n$. The $\Lloc^r$ convergence is straightforward, so we only need to control its derivatives. Proceeding with the usual argument, we obtain
\begin{align*}
\|\partial_j  \left(\big[a_\eps - (a*\rho_\eps) \big] (f*\rho_\eps)\right)\|_{L^r(K)} 
& \leq \|a_\eps - (a*\rho_\eps)\|_{W^{1,p}(K)} \|f\|_{L^q(K_1)} \\ &\quad + \|a_\eps - (a*\rho_\eps)\|_{L^p(K)}  \|\partial_j(f*\rho_\eps)\|_{L^q(K)}.    
\end{align*}
It is therefore enough to show that $\|a*\rho_\eps- a_\eps\|_{L^p(K)} \leq c_K \eps$, for some constant $c_K>0$ depending on $a$ and $K$, and that $\|a*\rho_\eps - a_\eps\|_{W^{1,p}(K)}\to 0$ as $\eps \to 0$. Both assertions follow from the hypothesis and \Cref{lemma:RateConvSobolev}, together with standard calculations. Then, by \Cref{lemma:AuxFriedrichswp}, $\big[a_\eps - (a*\rho_\eps) \big] (f*\rho_\eps) \to 0$ in $\Wloc^{1,r}(\RR^n)$ as $\eps \to 0$.
\end{proof}

\subsection{Scalar curvature}
This section focuses on establishing the convergence between the scalar curvature $R[g\star_M \rho_\eps]$ of the regularised metric $g\star_M \rho_\eps$ and the mollified scalar curvature $R[g] \star_M \rho_\eps$ of the non-smooth metric $g$. This result is a key ingredient for controlling the negative part of $R[g\star_M \rho_\eps]$ when $g$ has distributionally nonnegative scalar curvature, enabling the application of the conformal method.

The following result adapts the argument from \cite[Corollary 3.4]{Calisti2025}, where Friedrichs' Lemma played a crucial role in establishing the convergence of the regularised Ricci tensor for $C^{0,1}$ metrics. In our context, we extend this approach to derive a similar result for the scalar curvature when $g$ has lower regularity. 

\begin{proposition}\label{prop:ConvofRs}
    Let $g \in C^0 \cap \Wloc^{1,p}$, $p \in [2,\infty)$, then $R[g] \star_M \rho_\eps - R[g \star_M \rho_\eps ] \to 0$ in $\Lloc^{p/2}(M)$.
\end{proposition}
    \begin{proof}
    In this proof, let us denote $g \star_M \rho_\eps$ as $g_\eps$. Recall that, in coordinates, the distributional scalar curvature is given by
    \begin{equation}
        R[g] = \partial_k V_g^k + F_g,    
    \end{equation}
    where the vector field $V_g$ and the scalar field $F_g$ are given by \eqref{eq:V_defSc_W1n} and \eqref{eq:F_defSc_W1n}, respectively. It therefore suffices to show the convergence of the following terms
    \begin{equation}\label{eq:term1}
        F_{g_\eps} - F_g \star_M \rho_\eps \longrightarrow 0 \text{ in } \Lloc^{p/2},
    \end{equation} 
    and
    \begin{equation}\label{eq:term3}
        \partial_k V^k_{g_\eps}- \partial_k \left(V^k_g\right) \star_M \rho_\eps \longrightarrow 0 \text{ in } \Lloc^{p/2}.
    \end{equation}

    The term $F_g$ is schematically composed of sums of terms such as $g^{-1}\Gamma[g] \Gamma[g]$ and $\partial g^{-1} \Gamma[g]$, i.e., terms involving at most first weak derivatives of $g$. Similarly for $F_{g_\eps}$. Since $g$ is $\Wloc^{1,p}\cap C^0$, we have that $g_\eps \to g$ and $g_\eps^{-1} \to g^{-1}$ in $\Wloc^{1,p}(M)$, which implies $\Gamma_{kl}^m[g_\eps] \to \Gamma_{kl}^m[g]$ in $\Lloc^p(M)$, for all indices $k,l,m$. Therefore, by H\"{o}lder's inequality and continuity of $g$, all terms in \eqref{eq:term1} of the form $g_\eps^{-1}\Gamma[g_\eps] \Gamma[g_\eps]$ converge in $\Lloc^{p/2}(M)$ to $g^{-1} \Gamma[g] \Gamma[g]$. For terms of the form $\partial g^{-1} \Gamma[g]$, since they involve at most first derivatives of $g$, the same argument implies convergence in $\Lloc^{p/2}(M)$ of terms $\partial g_\eps^{-1}\Gamma[g_\eps]$ to $\partial g^{-1} \Gamma[g]$. This is sufficient to establish \eqref{eq:term1}.

For the term containing $V_g$, as discussed at the end of \Cref{subsection:DistCurvature}, it suffices to work in coordinates. Given a coordinate chart $(\psi, U)$, expressing $R[g\star_M \rho_\eps]$ and $R[g] \star_M \rho_\eps$ in local coordinates reduces the convergence problem to showing that $$
\mathfrak{g}_\eps^{ij} \mathfrak{g}_\eps^{ks} (\eta\partial_s \mathfrak{g}_{lm})* \rho_\eps - \left(\mathfrak{g}^{ij} \mathfrak{g}^{ks} \eta\partial_s \mathfrak{g}_{lm} \right) * \rho_\eps \longrightarrow 0 \text{ in } \Wloc^{1,p/2},
$$
where $\mathfrak{g}:=(\psi_*g)$ and $\mathfrak{g}_\eps:=(\psi_*g_\eps)$, and $\eta$ is any bump function compactly supported in $\psi(U)$, for all indices $i,j,k,s,l,m$.

The above expression corresponds exactly to the relevant terms in the proof of \cite[Lemma 4.6]{Gra:20} with an extra term given by the inverse of the metric. Since $C^0 \cap W^{1,p}$ is a Banach algebra, a standard computation reveals that $g^{ij} g^{ks} - g_\eps^{ij}g_\eps^{ks} \to 0$ in $\Wloc^{1,p}(M)$ and in $\Cloc^0\left(M\right)$ as $\eps \to 0$. From \Cref{col:inv_converg} (see below), for any compact $K$ there exist constants $c_K$ and sufficiently small $\eps$ such that
\begin{equation}\label{eq:convrate_inv}
    \|g^{ij} - g^{ij}_\eps\|_{L^p(K)} \leq c_K \eps,
\end{equation}
and it is straightforward to verify that this also implies
\begin{equation}
    \|g^{ij} g^{ks} - g_\eps^{ij}g^{ks}_\eps\|_{L^p(K)} \leq C_K \eps,
\end{equation}
for some $C_K>0$. Setting $a = \mathfrak{g}^{ij} \mathfrak{g}^{ks}$, $a_\eps = \mathfrak{g}_\eps^{ij} \mathfrak{g}_\eps^{ks}$, the $\Wloc^{1,p/2}$ convergence follows from Friedrichs' Lemma (see \Cref{cor:ConvFriedPlusApproximations}).
\end{proof}

\begin{corollary}
\label{col:inv_converg}
    Let $g \in C^0\cap \Wloc^{1,p}$, $p \in [1,\infty)$, for any compact $K \subset M$ there exists $c_K>0$ and $\eps_0(K)>0$ such that 
    \begin{equation}
        \|g^{-1} - (g \star_M \rho_\eps)^{-1}\|_{L^p(K)} \leq c_K \eps,
    \end{equation}
    for all $\eps < \eps_0(K)$. Moreover, we also have the following convergences
    \begin{equation}
    g^{-1} - (g \star_M \rho_\eps)^{-1} \longrightarrow 0 \text{ as } \eps \to 0, \text{ in } \Wloc^{1,p} \text{ and } \Cloc^0.
    \end{equation}
\end{corollary}
\begin{proof}
This proof is a straightforward adaptation of \cite[Corollary 4.3]{Gra:20}. Given that $g\star_M \rho_\eps$ converges locally uniformly to $g$, the net is locally uniformly bounded and locally uniformly non-degenerate. Using this fact together with \Cref{lemma:RateConvSobolev}, the result follows from \eqref{eq:ConvRateMollinLp} and the general formula $A^{-1} = \frac{1}{\text{det} A} C(A)^T$, where $C(A)$ denotes the matrix of cofactors for $A$. Although the determinant involves product entries, the convergence still has order $p$, as a consequence of the continuity of $g$.
\end{proof}
  
Before proceeding, we fix some notation. Let $f:M \to \RR$ be a function, and we use $f_+:M \to [0, +\infty)$ and $f_-:M \to [0, +\infty)$ to represent its positive and negative parts, respectively, i.e., $f = f_+ - f_-$. Let $g_\eps$ denote the family of metrics from \Cref{lemma:FamilyOfMetrics}, the following result shows that $R[g_\eps]_-$ converges to zero in $L^{p/2}(M)$ as $\eps \to 0$, provided that $g$ has nonnegative distributional scalar curvature (see \Cref{def:PosDistributions}). 

\begin{proposition}\label{prop:NegScalConv}
       Let $g$ be a complete, asymptotically flat $C^0 \cap \Wloc^{1,p}$, $p \in[2,\infty)$, metric, which is smooth outside a compact set. Assume that $R[g] \geq 0 $ in $\D^\prime(M)$. Then the negative part of the scalar curvature of the metric $g_\eps$ satisfies 
    $$\|R[g_\eps]_-\|_{L^{p/2}(M,g)} \longrightarrow 0  \text{ as } \eps \to 0.$$
\end{proposition}
\begin{proof}
This proof primarily relies on \Cref{prop:ConvofRs}, which was originally derived using $g \star_M \rho_\eps$ instead of $g_\eps$. Since both constructions are based on mollifiers combined with partitions of unity, it is clear that the arguments applies to $g_\eps$ by considering a localised regularisation $R[g] \star_K \rho_\eps$ of $R[g]$ on a compact set $K$, analogous to the construction of $g_\eps$ (see \Cref{lemma:FamilyOfMetrics}). In particular, $R[g] \star_K \rho_\eps = R[g]$ outside a compact set $K_0$.

Recall that if $R[g]\geq 0$ in $\D^\prime(M)$ $\implies R[g]\star_K\rho_\eps \geq 0$. Splitting the scalar curvature of the regularised metric into positive and negative parts, we have that $R[g_\eps] = R[g_\eps]_+ - R[g_\eps]_-$ and $K_0 \subset M$ contains the support of $R[g_\eps]_-$, since $g_\eps =g$ and $R[g_\eps] = R[g]$ away from $K_0$. Using that $R[g]\star_K\rho_\eps \geq 0 $, the scalar curvature can be bounded from below as
\begin{equation}
R[g_\eps] = \left(R[g_\eps] - R[g]\star_K \rho_\eps\right) + R[g]\star_K \rho_\eps \geq - \left| R[g_\eps] - R[g]\star_K \rho_\eps\right|,      
\end{equation}
which yields $\left|R[g_\eps]_-\right| \leq \left|R[g_\eps] - R[g] \star_K \rho_\eps \right|$. Taking all norms with respect to $g$, we have
\begin{align*}
    \|R[g_\eps]_- \|_{L^{p/2}(M)}  & = \| R[g_\eps]_- \|_{L^{p/2}(K_0)}\\
    & \leq  \|R[g_\eps] - R[g] \star_K \rho_\eps \|_{L^{p/2}(K_0)}.
\end{align*}
Hence, the desired conclusion follows directly from \Cref{prop:ConvofRs}.
\end{proof}

\begin{remark}\label{prop:NegScalConvMetricEps}
    If the Sobolev space $W^{1,p}(M)$ is defined with respect to $g$, and considering the family of metrics $g_\eps$ constructed in \Cref{lemma:FamilyOfMetrics}, then by \Cref{lemma:bilinearmetric} we have the estimate $\|R[g_\eps]_-\|_{L^{p/2}(M, g_\eps)} \leq  \rho(\eps)^n \|R[g_\eps]_-\|_{L^{p/2}(M, g)}$ which converges to zero by \Cref{prop:NegScalConv} and $\rho(\eps) \to 1$ as $\eps \to 0$. This convergence result will be useful in the subsequent section.
\end{remark}

\section{Positive mass theorem}\label{section:PMT}

\subsection{Nonnegativity of the ADM mass}
In this section, we establish our main theorem by constructing a sequence of smooth metrics $\gtilde_\eps$ converging to $g$, with $R[\gtilde_\eps]\geq0$. By the smooth positive mass theorem, \Cref{theo:OriginalPMT}, this implies $m(\gtilde_\eps)\geq 0$, for all $\eps>0$. It therefore remains to show that $m(\gtilde_\eps) \to m(g)$ as $\eps \to 0$. We begin with the approximating metrics $g_\eps$ from \Cref{lemma:FamilyOfMetrics}, and a key challenge is that $g_\eps$ need not have nonnegative scalar curvature. To address this, we follow the approach introduced by \cite{Miao2002}, which consists in performing a conformal deformation of the smooth metrics $g_\eps$ to obtain a family of metrics $\gtilde_\eps$ with nonnegative scalar curvature. This conformal deformation is carried out using the following lemma, originally proved by Schoen and Yau \cite{PMTSchoenYau}.

\begin{lemma}[\cite{PMTSchoenYau, Miao2002}]\label{lemma:MiaoConfMetric}
    Let $g$ be a $C^2$ complete, asymptotically flat metric on $M^n$ and a compactly supported function $f$. There exists a number $\eps_0(g)>0$ so that if
    \begin{equation}\label{eq:ConfMetricCondition}
    \left( \int_M |f_-|^{\frac{n}{2}} \,\dg \right)^\frac{2}{n} < \eps_0(g),    
    \end{equation}
    then
\begin{equation}\label{eq:MiaoPDESystem}
\left\{ \begin{aligned} 
  -4\frac{(n-1)}{(n-2)} \Delta_g u + f u & = 0\\
  \lim_{x \to \infty} u & = 1
\end{aligned} \right.
\end{equation}
has a $C^2$ positive solution $u$ on $M$ so that $u = 1 + \frac{A}{|x|^{n-2}} + \omega$ for some constant $A$ and some function $\omega$, where $\omega = O_2\left(|x|^{1-n}\right)$.
\end{lemma}

A standard argument, as presented in \cite{Miao2002}, ensures that we can conformally deform $g_\eps$ to achieve nonnegative scalar curvature without significantly changing the mass. We restate this argument here for completeness. Let $u$ be a smooth positive function on $M^n$. If $\gtilde =  u^\frac{4}{n-2} g$, we have that 
\begin{equation}
    R[\gtilde] = u^{-\frac{n+2}{n-2}} L_g u,
\end{equation} 
where $L_g$ is the \textit{conformal Laplacian} of $g$ and it is given by
\begin{equation}
    L_g u := -c_n \Delta_g u + R[g]u,
\end{equation}
where $c_n = 4\frac{(n-1)}{(n-2)}$. 

Inspecting the proof of the lemma in \cite{PMTSchoenYau}, we observe that $\eps_0(g_\eps) = 1/C(g_\eps)$, where $C(g_\eps)$ is the Sobolev constant of the metric $g_\eps$. Set $f_\eps = - \frac{1}{c_n}R[g_\eps]_-$. Condition \eqref{eq:ConfMetricCondition} can then be rewritten as
\begin{equation}\label{eq:exist_sol}
C(g_\eps) \left(\int_M \frac{1}{c_n}\big|R[g_\eps]_-\big|^{n/2}\,\dgeps\right)^{2/n} \leq 1.    
\end{equation}
By \Cref{prop:NegScalConvMetricEps} and \Cref{prop:SobolevConstant}, there exists $\eps_0>0$ such that, for all $\eps < \eps_0$, the above condition is satisfied.

It then follows from \Cref{lemma:MiaoConfMetric} that, for each $\eps < \eps_0$, there exists a $C^2$, positive solution $u_\eps$ of
\begin{equation}\label{eq:PDEnegativeS}
   - \Delta_{g_\eps}  u_\eps + f_\eps u_\eps = 0,
\end{equation}
satisfying $\lim_{x \to \infty} u_\eps(x) = 1$. Therefore, we can define a Riemannian metric $\gtilde_\eps$ by the conformal rescaling
\begin{equation}\label{eq:DefConfMetricEps}
\gtilde_\eps = u_\eps^\frac{4}{n-2} g_\eps.
\end{equation}
In particular, the nonnegativity of the scalar curvature of $\gtilde_\eps$ follows from the conformal transformation formula for the scalar curvature, 
\begin{align}\label{eq:ConfScalarCurv}
   R[\gtilde_\eps] = u_\eps^{-\frac{n+2}{n-2}} L_{g_\eps} (u_\eps) = u_\eps^{-\frac{n+2}{n-2}} \big( R[g_\eps]_+ u_\eps \big) = u_\eps^{\frac{4}{2-n}} R[g_\eps]_+\geq 0.
\end{align}
The proof of our main theorem now requires control of the solutions $u_\eps$ as $\eps \to 0$. This is crucial because the relationship between the masses of $g_\eps$ and $\gtilde_\eps$ depends on $u_\eps$ and its first derivatives. Henceforth $g_\eps$ denotes the family of metrics from \Cref{lemma:FamilyOfMetrics}. To establish $m(\gtilde_\eps) \to m(g_\eps)$, we present the following proposition.

\begin{proposition}\label{prop:ConvGraduetc}
    Let $g \in C^0 \cap \Wloc^{1,n}$ be a complete, asymptotically flat Riemannian metric on $M^n$, smooth outside a compact set, and assume that $R[g] \geq 0 $ in $\D^\prime(M)$. Let $u_\eps$ be a  solution of \eqref{eq:PDEnegativeS}. Then
    $$\| u_\eps - 1 \|_{L^{n^*}(M, g_\eps)}+\| \nabla^{g_\eps} u_\eps \|_{L^{2}(M, g_\eps)} \longrightarrow 0 \text{ as } \eps \to 0, $$
where $n^* = \frac{2n}{n-2}$.
\end{proposition}

\begin{proof}
   To establish the current result, we adopt the procedure outlined in \cite{grant2014positive}. Let $\eps>0$ be such that $C(g_\eps) \|R[g_\eps]_-\|_{L^{n/2}(M,g_\eps)}<1$, $R[g_\eps]_-$ compactly supported, and recall that $g_\eps$ agree with $g$ outside of a compact set $K_0$ by \Cref{lemma:FamilyOfMetrics}. 
   
    Define $w_\eps := u_\eps -1$ and $f_\eps := -\frac{1}{c_n}R[g_\eps]_-$. Then $w_\eps$ satisfies
   \begin{equation}
       -\Delta_{g_\eps} w_\eps + f_\eps w_\eps + f_\eps = 0,
   \end{equation}
   Since $w_\eps=O_2(|x|^{2-n})$ and decays sufficiently fast, we can multiply both sides by $w_\eps$, integrate over $M$, and apply the divergence theorem to obtain
\begin{align}
    \int_M |\nabla^{g_\eps} w_\eps|^2 \,\dgeps= -\int_M f_\eps w_\eps^2 + f_\eps w_\eps \,\dgeps.
\end{align} 
Let $n^* = \frac{2n}{n-2}$. It follows from the above identity and H\"{o}lder's inequality that 
\begin{align*}
\int_M |\nabla^{g_\eps} w_\eps|^2 \,\dgeps & =  \int_M f_\eps w_\eps^2 + f_\eps w_\eps\,\dgeps \\
&  \leq  \|f_\eps\|_{L^{n/2}(M, g_\eps)} \|w_\eps\|_{L^{n^*}(M, g_\eps)}^2  + \|f_\eps\|_{L^{n/2}(M, g_\eps)} \|1\|_{L^{n^*}(\operatorname{supp}f_\eps, g_\eps)}  \|w_\eps\|_{L^{n^*}(M, g_\eps)}\\
& \leq  \|f_\eps\|_{L^{n/2}(M, g_\eps)} \|w_\eps\|_{L^{n^*}(M, g_\eps)}^2  + \|f_\eps\|_{L^{n/2}(M, g_\eps)} \Vol_{g_\eps}(\operatorname{supp} f_\eps)^{1/n^*} \|w_\eps\|_{L^{n^*}(M, g_\eps)}.
\end{align*}
The Sobolev inequality yields the following estimate 
\begin{align*}
    \|w_\eps\|_{L^{n^*}(M, g_\eps)}^2& \leq C(g_\eps) \|\nabla^{g_\eps} w_\eps\|_{L^{2}(M, g_\eps)}^2,
\end{align*}
where $C(g_\eps)$ is chosen to be the smallest constant satisfied by the inequality. Combining the last two inequalities, we derive
\begin{align*}
     \|w_\eps\|_{L^{n^*}(M, g_\eps)}^2 
    & \leq\beta_\eps \|w_\eps\|_{L^{n^*}(M, g_\eps)}^2  + \beta_\eps \Vol_{g_\eps}(\operatorname{supp} f_\eps)^{1/n^*} \|w_\eps\|_{L^{n^*}(M, g_\eps)},
\end{align*}
where $\beta_\eps = C(g_\eps) \|f_\eps\|_{L^{n/2}(M, g_\eps)}$. By \Cref{prop:SobolevConstant}, we have the bound $C(g_\eps) \leq \rho(\eps)^{n} C(g)$. Hence, we have that $ \beta_\eps = C(g_\eps)\frac{1}{c_n} \|R[g_\eps]_-\|_{L^{n/2}(M, g_\eps)} \to 0 $ as $\eps \to 0$, by \Cref{prop:NegScalConvMetricEps}. For $\eps$ sufficiently small, we guarantee that $\beta_\eps<1$, and therefore deduce that
\begin{equation}\label{w_nstar_estimate}
     \|u_\eps-1\|_{L^{n^*}(M, g_\eps)} = \|w_\eps\|_{L^{n^*}(M, g_\eps)} \leq \frac{\beta_\eps}{1-\beta_\eps} \Vol_{g_\eps}(\operatorname{supp} f_\eps)^{1/n^*}.
\end{equation}
Since $\|f_\eps\|_{L^{n/2}(M, g_\eps)} \to 0$ by \Cref{prop:NegScalConvMetricEps} and $\operatorname{supp}f_\eps$ is contained in a fixed large compact set, it follows that $\| \nabla^{g_\eps} u_\eps \|_{L^{2}(M, g_\eps)}$ and $\|u_\eps-1\|_{L^{n^*}(M, g_\eps)}$ approach zero as $\eps \to 0$.
\end{proof}

\begin{lemma}\label{lemma:ADMmassConv}
  Let $g \in C^0 \cap \Wloc^{1,n}$ be a complete, asymptotically flat Riemannian metric on $M^n$, smooth outside a compact set, and assume that $R[g] \geq 0 $ in $\D^\prime(M)$. Let $\gtilde_\eps$ denote the conformally rescaled metric defined in \eqref{eq:DefConfMetricEps}. Then $\gtilde_\eps$ is a complete, asymptotically flat metric with nonnegative scalar curvature, and its ADM mass satisfies 
  $$
  m(\gtilde_\eps) \longrightarrow m(g) \quad \text{ as } \eps \to 0.
  $$
\end{lemma}
\begin{proof}
    By \eqref{eq:ConfScalarCurv}, the scalar curvature $R[\gtilde_\eps]$ is nonnegative. Note that $g_\eps$ is complete for $\eps>0$, since $g$ is complete as a metric space by hypothesis. Moreover, because $u_\eps \in C^2$, $u_\eps>0$, and $\lim_{x \to \infty} u_\eps(x) = 1$, the conformally rescaled metrics $\gtilde_\eps$ are complete. Additionally, \Cref{lemma:MiaoConfMetric} shows that $\gtilde_\eps$ is asymptotically flat.
    
    As described in \cite[Lemma 4.2]{Miao2002}, a standard calculation reveals that the masses of $(M, \gtilde_\eps)$ and $(M, g_\eps)$ are related according to the following
    \begin{equation}
        m(\gtilde_\eps) = m({g_\eps}) + 2 A_\eps,
    \end{equation}
where $A_\eps$ is given by the expansion $u_\eps = 1 + \frac{A_\eps}{|x|^{n-2}} + O_2\left(|x|^{1-n}\right)$ (see \Cref{lemma:MiaoConfMetric}). Multiplying \eqref{eq:PDEnegativeS} by $u_\eps$, integrating by parts, and using the decay rate of $u_\eps$, we have that
\begin{equation}
    (2-n) \omega_{n-1} A_\eps = \int_M |\nabla^{g_\eps} u_\eps|^2 - \frac{1}{c_n}R[g_\eps]_- u_\eps^2 \,\dgeps,
\end{equation}
where $\omega_{n-1}$ is the area of the $n-1$ dimensional unit sphere in $\RR^n$. It follows that
\begin{equation}\label{eq:mass_formula}
     m(\gtilde_\eps) = m({g_\eps}) - \frac{2}{(n-2)\omega_{n-1}} \int_M |\nabla^{g_\eps} u_\eps|^2 - \frac{1}{c_n}R[g_\eps]_- u_\eps^2\,\dgeps.
\end{equation}   
Then, it remains to show that the last term on the right-hand side goes to zero as $\eps \to 0$. Let $K$ and $K_0$ be the compact sets from \Cref{lemma:FamilyOfMetrics}. By H\"{o}lder's inequality, we have
\begin{equation}
 \left| \int_M R[g_\eps]_- u_\eps^2 \,\dgeps \right| = \left| \int_{K_0} R[g_\eps]_- u_\eps^2\,\dgeps\right| \leq \|R[g_\eps]_-\|_{L^{n/2}(K_0, g_\eps)} \|u_\eps\|_{L^{n^*}(K_0, g_\eps)}^2.   
\end{equation}
Furthermore, notice that
\begin{align*}
    \|u_\eps\|_{L^{n^*}(K_0, g_\eps)} & \leq \|1\| _{L^{n^*}(K_0, g_\eps)} + \|u_\eps-1\|_{L^{n^*}(K_0, g_\eps)}\\
    & \leq \Vol_{g_\eps}(K_0)^{1/n^*} + \|u_\eps-1\|_{L^{n^*}(M, g_\eps)},
\end{align*}
which approaches $\Vol_{g}(K_0)^{1/n^*}$ as $\eps \to 0$, by \Cref{prop:ConvGraduetc}.
Clearly, this volume is finite and \Cref{prop:NegScalConvMetricEps} leads to 
\begin{equation}
\int_M R[g_\eps]_- u_\eps^2\,\dgeps \longrightarrow 0 \text{ as } \eps \to 0,    
\end{equation}
and, by \Cref{prop:ConvGraduetc},
\begin{equation}
\int_M |\nabla^{g_\eps}u_\eps|^2\,\dgeps \longrightarrow 0 \text{ as } \eps \to 0.   
\end{equation}
As a result, since $m(g_\eps)=m(g)$, we conclude that
\begin{equation}
\lim_{\eps\to 0} m(\gtilde_\eps) = \lim_{\eps\to 0} m(g_\eps) = m(g).
\end{equation}
\end{proof}

\begin{theorem}\label{theo:PMT_ineq}
    Let $M^n$, $n\geq 3 $, be a smooth manifold endowed with a complete, asymptotically flat Riemannian metric $g \in C^0\cap \Wloc^{1,n}$, smooth outside a compact set. If $R[g]\geq0$ in $\D^\prime$, then $m(g)\geq0$.
\end{theorem}
\begin{proof}
    The statement is obtained by applying the smooth positive mass theorem, \Cref{theo:OriginalPMT}, together with \Cref{lemma:ADMmassConv}. In particular, since $m(\gtilde_\eps)\geq 0$, it follows $m(g)\geq 0$.
\end{proof}
\subsection{Rigidity case}

In this section, we prove the rigidity case for the ADM mass, i.e., that $m(g) = 0$ implies that $(M^n, d_g)$ is isometric to $(\RR^n, d_\delta)$ as a metric space. Moreover, if $g \in W^{1,p}_\loc$, $p>n$, then there exists a $C^{1,1-\frac{n}{p}}$ Riemannian isometry between $(M, g)$ and $(\RR^n, \delta)$. The proof follows the usual blueprint of first proving scalar flatness and Ricci flatness and then applying volume comparison arguments. The Riemannian isometry follows from a low-regularity Myers-Steenrod theorem.

The proofs of scalar flatness and Ricci flatness rely crucially on the improved estimates for the regularised scalar curvature obtained in \Cref{prop:NegScalConv} and on the technical \Cref{lemma:ScalarLowerBound}. Both proofs proceed by contradiction, using conformal transformations to construct a family of metrics satisfying the assumptions of the smooth positive mass theorem but having negative ADM mass. For the Ricci flatness proof, the argument is inspired by Jiang-Sheng-Zhang \cite[Lemma 4.1]{Jiang2022}, and crucially uses \Cref{prop:ConvofRs} to conclude that $R[g_\eps] \to 0$ in $L^1(M)$ if $R[g]=0$. Together with a careful analysis of the conformal operator, this convergence yields Ricci flatness globally, rather than only outside a compact set, as obtained in \cite[Lemma 4.1]{Jiang2022}. From Ricci flatness to isometry, we use the equivalence between distributional Ricci curvature lower bounds and $\RCD$-spaces established in \cite{mondino2025}, together with the volume comparison argument for $\RCD$-spaces from \cite{Jiang2022}.

In what follows, we prove the technical \Cref{lemma:ScalarLowerBound}, and then proceed with the rigidity proof. We also introduce $\RCD$-spaces after proving Ricci flatness.

\begin{lemma}\label{lemma:ScalarLowerBound}
    Let $g \in C^0 \cap \Wloc^{1,n}$ be a complete, asymptotically flat Riemannian metric on $M^n$, smooth outside a compact set. Let $g_{\eps,t} = g_\eps + t q$, where $q$ is a compactly supported $(0,2)$ symmetric tensor $q$ and $g_\eps$ is the mollified metric in \Cref{lemma:FamilyOfMetrics}. Consider $C^1$ functions $w,v \colon M \to \RR$ with $dw, dv \in L^2(M)$ satisfying $w, v \in L^{\frac{2n}{n-2}}(M)$ and $w,v = O_1(|x|^{-\tau})$, where $\tau>(n-2)/2$. Then, there exist $\eps_0,t_0>0$ and constants $C_0, C_1,C_2>0$ independent of $\eps$ and $t$ such that for all $\eps \in (0, \eps_0)$ and $t \in [0, t_0)$, we have 
    \begin{equation}\label{eq:linear_control}
        \left| \int  R[g_{\eps,t}] w \, d\mu_{\eps,t} \right| \leq \left( C_0  + C_1 t +C_2 t^2 \right) \|\nabla^{g_{\eps,t}} w\|_{L^2(M, g_{\eps,t})} ,
    \end{equation}
    and
     \begin{equation} \label{eq:quadratic_control}
        \left| \int  R[g_{\eps,t}] w v \, d\mu_{\eps,t} \right| \leq \left( C_0  + C_1 t +C_2 t^2 \right) \|\nabla^{g_{\eps,t}} w\|_{L^2(M, g_{\eps,t})} \|\nabla^{g_{\eps,t}} v\|_{L^2(M, g_{\eps,t})},
    \end{equation}
   
    Moreover, if $R[g] = 0$, then, after possibly decreasing $\eps_0$ and $t_0$, there exists $A(\eps)\geq0$ satisfying $A(\eps) \to 0$ as $\eps \to 0$, such that $C_0$ may be replaced by $A(\eps)$ in both estimates above.
\end{lemma}
    \begin{proof}
        
     Recall the following formula \cite[Exercise 1.18]{L:19} (or \cite[Proposition 4]{BrendleMarques2011}),
    \begin{equation} \label{eq:VarScalarCurv_complete2}
        R[g_{\eps,t}] = R[g_\eps] + t\restr{DR}{g_\eps}(q) + t^2 B_{\eps,t}(q),
    \end{equation}
    where
    \begin{equation}\label{eq:VarScalarCurv2}
        \restr{DR}{g_\eps}(q) = \operatorname{div}_{g_\eps}(\operatorname{div}_{g_\eps}(q)) - \Delta_{g_\eps} \tr_{g_\eps}(q)  -  \langle \Ric[g_\eps] , q \rangle_{g_\eps},
    \end{equation}
    and $B_{\eps,t}(q)$ is the sum of a divergence of a compactly supported vector field $\operatorname{div}_{g_\eps} U_{\eps,t}$ where $U_{\eps,t}$ is a contraction of $q*\nabla^{g_\eps} q$ with $g^{-1}_\eps$ and $g^{-1}_{\eps,t}$, a term quadratic in $\nabla^{g_\eps} q$, and a term $\Ric[g_\eps]*q^2$ (with similar contraction).

    By a partition of unity argument, it suffices to prove the claim for the case $M = \RR^n$. We also assume that $g_\eps = g$ outside a compact set as constructed in \Cref{lemma:FamilyOfMetrics}. We write $R[g_\eps] = \partial_k V^k[g_\eps] + Q[g_\eps]$, as in \eqref{eq:dist_scalar_W1n}. Although $w$ and $v$ are not necessarily compactly supported, integration by parts yields
    \begin{equation}
        \int R[g_\eps] w v \,d\mu_{g_\eps} = \int -V^k[g_\eps] \partial_k(wv \sqrt{|g_\eps|}) + Q[g_\eps]  wv \sqrt{|g_\eps|} \, dx,
    \end{equation}
    where the boundary term vanishes at infinity due to the sufficient decay rate of $V[g_\eps]$ and $w,v$, i.e. $V[g_\eps] = O(|x|^{-1-\tau})$ and $w, v = O_1(|x|^{-\tau})$, where $\tau>(n-2)/2$.

    Note that $Q[g_\eps] \in L^{p}(M)$ for all $p \in [1,n/2]$, since $Q[g_\eps] \approx (\partial g_\eps)^2$, $\partial g_\eps \in L^{n}_\loc$ and $(\partial g_\eps)^2 = O(|x|^{-2-2\tau})=O(|x|^{-n-\delta})$ for some $\delta>0$. We can also conclude that $V^k[g_\eps] \in L^{p}(M)$ for all $p \in [2,n]$, using the fact that $V^k[g_\eps] \approx \partial g_\eps$, $\partial g_\eps \in L^{n}_\loc$ and $\partial g_\eps = O(|x|^{-1-\tau})=O(|x|^{-n/2-\delta})$ for some $\delta>0$.

    Then, for any $Q$ such that $Q \in L^{p}(M)$  for all $p \in [1,n/2]$, and $w,v \in C^1$ functions satisfying the hypothesis of the lemma, we have that
    \begin{align}
    \|Q w v \|_{L^1}& \leq \|Q w\|_{L^{\frac{2n}{n+2}}}  \|v \|_{L^{n^*}} \leq \|Q\|_{L^{\frac{n}{2}}}  \|w \|_{L^{n^*}} \|v \|_{L^{n^*}}
     \leq C(h)^2 \|Q\|_{L^\frac{n}{2}}  \|d w \|_{L^{2}}  \|d v \|_{L^{2}},
    \end{align}
    where $n^* = 2n/(n-2)$, and we have applied the Sobolev inequality \eqref{eq:SobolevConstant0} twice with respect to the background metric $h$ to obtain $C(h)^2$. Similarly for the linear case, we have
    \begin{align}
    \|Q w\|_{L^1}& \leq C(h) \|Q\|_{L^{\frac{2n}{n+2}}}      \|d w \|_{L^{2}}.
    \end{align}

    For any $V$ such that $V \in L^p(M)$ for all $p \in [2,n]$, we have
    \begin{align*}
        \|V \partial(w v)\|_{L^1} & \leq \|V v\|_{L^2} \| \partial w \|_{L^2} + \|V w\|_{L^2} \| \partial v\|_{L^2}\\
        & \leq \|V\|_{L^n} \left( \|v \|_{L^{n^*}}  \|\partial w \|_{L^2} +  \|w \|_{L^{n^*}}  \|\partial v\|_{L^2}\right)\\
            & \leq 2 C(h) \|V\|_{L^n} \left(\| d w \|_{L^2} \|dv\|_{L^2}\right),
    \end{align*}
    and
    \begin{align*}
        \|V w v \partial(\sqrt{|g_\eps|})\|_{L^1} 
        &  \leq \|Vv \|_{L^2}  \|\partial(\sqrt{|g_\eps|})\|_{L^n} \|w \|_{L^{n^*}} \\
        &  \leq \|V\|_{L^n}  \|\partial(\sqrt{|g_\eps|})\|_{L^n} \|w \|_{L^{n^*}} \|v \|_{L^{n^*}} \\ 
        & \leq C(h)^2\|V\|_{L^n}  \|\partial(\sqrt{|g_\eps|})\|_{L^n}  \left(\|d w \|_{L^2} \|d v\|_{L^2}\right),
    \end{align*}
    where we have applied the Sobolev inequality \eqref{eq:SobolevConstant0} again. Note that $\|\partial(\sqrt{|g_\eps|})\|_{L^n}$ can be bounded independently of $\eps$ for sufficiently small $\eps$ since $g_\eps = g$ outside a compact set and $g_\eps \to g$ in $C^0_\loc \cap W^{1,n}_\loc$. For the linear case, the preceding estimate yields
    \begin{align}
        \|V \partial w\|_{L^1}  \leq \|V\|_{L^2} \| d w \|_{L^2}\, \text { and }\,
        \|V w \partial(\sqrt{|g_\eps|})\|_{L^1} 
          \leq C(h) \|V\|_{L^2}  \|\partial(\sqrt{|g_\eps|})\|_{L^n} \|d w \|_{L^2}.
    \end{align}
    Therefore, from the above computations but with respect to the norm induced by $g_\eps$, we obtain
    \begin{align}\label{eq:Scalar_UpperBound1}
           \left| \int  R[g_{\eps}] w v \, d\mu_{\eps}\right| \leq  C(\eps)  \|\nabla^{g_\eps} w\|_{L^2(M, g_\eps)} \|\nabla^{g_\eps} v\|_{L^2(M,g_\eps)},
    \end{align}
    and
    \begin{align}\label{eq:Scalar_UpperBound2}
        \left| \int  R[g_{\eps}] w \, d\mu_{\eps}\right| \leq  C(\eps)  \|\nabla^{g_\eps} w\|_{L^2(M, g_\eps)},
    \end{align}
    where
    \begin{equation}
         C(\eps) \colonequals C \left(\sum_{k=1}^n\|V^k[g_\eps]\|_{L^n\cap L^2(M, g_\eps)} +  \|Q[g_\eps]\|_{L^\frac{n}{2} \cap L^\frac{2n}{n+2}(M, g_\eps)}\right),
    \end{equation}
    for some constant $C>0$ uniformly in $\eps$, for sufficiently small $\eps$. In particular, since $g_\eps \to g$ in $C^0_\loc \cap W^{1,n}_{\loc}$ and $g_\eps = g$ outside a compact set, the norms of $V^k[g_\eps]$ and $Q[g_\eps]$ are uniformly bounded in $\eps$, then there exists a constant $C_0$ uniformly in $\eps$ such that $C(\eps)\leq C_0$.

     Since $q$ is compactly supported, we can integrate by parts the divergence term in \eqref{eq:VarScalarCurv2}, which yields 
    \begin{equation}
        \int_{M}  \operatorname{div}_{g_\eps}(\operatorname{div}_{g_\eps}(q))  w v  \,d\mu_{g_\eps}  =  - \int_{M}  \langle \operatorname{div}_{g_\eps}(q), w \nabla^{g_\eps} v + v \nabla^{g_\eps} w  \rangle_{g_\eps} \,d\mu_{g_\eps},
    \end{equation}
    the second term can be estimated by
    \begin{align*}
        \left|\int_{M}  \langle \operatorname{div}_{g_\eps}(q), v \nabla^{g_\eps} w  \rangle_{g_\eps} \,d\mu_{g_\eps}\right|  & \leq \|\operatorname{div}_{g_\eps}(q)\|_{L^n(M, {g_\eps})} \left( \|v\|_{L^{n^*}(M, {g_\eps})} \|  \nabla^{g_\eps} w\|_{L^2(M, {g_\eps})}  \right)\\
        & \leq  C(g_\eps) \|\operatorname{div}_{g_\eps}(q)\|_{L^n(M, {g_\eps})}  \|  \nabla^{g_\eps} w\|_{L^2(M, {g_\eps})}  \|  \nabla^{g_\eps} v\|_{L^2(M, {g_\eps})},
    \end{align*}
    then
    \begin{equation}\label{eq:Claim1_divterm}
        \left|\int_{M}  \operatorname{div}_{g_\eps}(\operatorname{div}_{g_\eps}(q))  w v  \,d\mu_{g_\eps} \right| \leq 2 C(g_\eps) \|\operatorname{div}_{g_\eps}(q)\|_{L^n(M, {g_\eps})}  \|  \nabla^{g_\eps} w\|_{L^2(M, {g_\eps})}    \|  \nabla^{g_\eps} v\|_{L^2(M, {g_\eps})} ,
    \end{equation}
    where $C(g_\eps)$ is the Sobolev constant of $g_\eps$. 

    Analogously, for the Laplacian term in \eqref{eq:VarScalarCurv2}, we have that 
    \begin{align}
        \left| \int_{M}  \Delta_{g_\eps} \tr_{g_\eps}(q) w v \,d\mu_{g_\eps}\right| \leq & 2 C(g_\eps) \|\nabla^{g_\eps} \tr_{g_\eps}(q)\|_{L^n(M, {g_\eps})}  \|  \nabla^{g_\eps} w\|_{L^2(M, {g_\eps})}    \|  \nabla^{g_\eps} v\|_{L^2(M, {g_\eps})}.  \label{eq:Claim1_Lapterm}
    \end{align}
    Using a similar formula for the distributional Ricci as for the distributional scalar curvature \eqref{eq:dist_scalar_W1n} and estimating similarly as for the scalar curvature, we have 
    \begin{align}\label{eq:Claim1_Ricciterm}
        \left|\int_{M}  \langle \Ric[g_\eps] , w v q \rangle_{g_\eps}   \,d\mu_{g_\eps}\right|  \leq C  \| \nabla^{g_{\eps}} w\|_{L^2(M, g_\eps)} \|  \nabla^{g_{\eps}} v\|_{L^2(M, g_\eps)},
    \end{align}
    for some constant $C$ uniformly in $\eps$, for $\eps$ sufficiently small, and depending on $q$.

    The Ricci term in $B_{\eps,t}(q)$ \eqref{eq:VarScalarCurv_complete2}, can be treated exactly as above, but using $q^2$ instead. The quadratic terms $\nabla^{g_\eps} q$ in $B_{\eps,t}(q)$ \eqref{eq:VarScalarCurv_complete2} can be estimated by the Sobolev inequality \eqref{eq:SobolevConstant} as 
    \begin{align*}
    \left| \int_{M}  (\nabla^{g_\eps} q*\nabla^{g_\eps} q) w v \,d\mu_{g_\eps}\right|  & \leq \|\nabla^{g_\eps}q * \nabla^{g_\eps} q \|_{L^{\frac{n}{2}}(M, g_\eps)}  \|w\|_{L^{n^*}(M, g_\eps)} \|v\|_{L^{n^*}(M, g_\eps)}  \\
     & \leq C(g_\eps)^2 \|\nabla^{g_\eps}q * \nabla^{g_\eps} q \|_{L^{\frac{n}{2}}(M, g_\eps)} \|\nabla^{g_\eps} w\|_{L^2(M, g_\eps)}\|\nabla^{g_\eps} v\|_{L^2(M, g_\eps)},      
    \end{align*}
    and the estimate for the divergence term in $B_{\eps,t}(q)$ is similar to the estimate \eqref{eq:Claim1_divterm}.

    Using that $g_\eps$ and $g_{\eps,t}$ are uniformly bounded in $\eps$ and $t$, we obtain
    \begin{align}\label{eq:Claim1_Bkterm}
    \left| \int_{M}  B_{\eps,t}(q) w v \,d\mu_{g_\eps}\right|  \leq C  \|\nabla^{g_\eps} w\|_{L^2(M, g_\eps)}\|\nabla^{g_\eps} v\|_{L^2(M, g_\eps)},
    \end{align}
    for some constant $C$ uniformly in $\eps$.

    Now consider instead $w v d\mu_{g_{\eps,t}}/d\mu_{g_{\eps}} = \left(wd\mu_{g_{\eps,t}}/d\mu_{g_{\eps}}\right) v $, and note that
    \begin{align*}
        \|\nabla^{g_\eps} \left(w\frac{d\mu_{g_{\eps,t}}}{d\mu_{g_{\eps}}}\right)\|_{L^2(M, g_\eps)}& \leq \|\frac{d\mu_{g_{\eps,t}}}{d\mu_{g_{\eps}}}\|_{L^\infty} \|\nabla^{g_\eps} w\|_{L^2(M, g_\eps)} + \|w \nabla^{g_\eps} \left(\frac{d\mu_{g_{\eps,t}}}{d\mu_{g_{\eps}}}\right)\|_{L^2(M, g_\eps)}\\
        & \leq \|\frac{d\mu_{g_{\eps,t}}}{d\mu_{g_{\eps}}}\|_{L^\infty} \|\nabla^{g_\eps} w\|_{L^2(M, g_\eps)} + \|w\|_{L^{n^*}(M, g_\eps)} \|\nabla^{g_\eps} \left(\frac{d\mu_{g_{\eps,t}}}{d\mu_{g_{\eps}}}\right)\|_{L^n(M, g_\eps)}\\
        & \leq C \|\nabla^{g_\eps} w\|_{L^2(M, g_\eps)},
    \end{align*}
    where we have used that $g_{\eps,t} \to g$ in $C^0_\loc\cap W^{1,n}_\loc$ as $t \to 0$ and $\eps \to 0$, and the Sobolev inequality with respect to $g_\eps$ to obtain a constant $C>0$ uniformly in $\eps$, for $\eps$ and $t$ sufficiently small. 
    
    In order to obtain the linear estimate, we repeat the above computation with a single test function $w$ in place of the product $wv$. For the terms involving $q$, it requires $L^2$ and $L^{\frac{2n}{n+2}}$ bounds instead of $L^n$ and $L^{\frac{n}{2}}$, respectively. Since these terms are supported in a fixed compact set $\supp q$, the $L^p$ embeddings provide the required estimates. Then, combining \eqref{eq:Scalar_UpperBound1}, \eqref{eq:Claim1_divterm}, \eqref{eq:Claim1_Lapterm}, \eqref{eq:Claim1_Ricciterm}, and \eqref{eq:Claim1_Bkterm} (or the equivalent estimates for the linear case) and choosing $\eps$ and $t$ smaller if necessary, we obtain 
    \begin{equation}\label{eq:Scalar_CaseR0}
        \left| \int  R[g_{\eps,t}] w v \, d\mu_{g_{\eps,t}} \right| \leq \left( C_0  + C_1 t +C_2 t^2 \right) \|\nabla^{g_{\eps,t}} w\|_{L^2(M, g_{\eps,t})}\|\nabla^{g_{\eps,t}} v\|_{L^2(M, g_{\eps,t})},
    \end{equation}
    and
    \begin{equation}\label{eq:Scalar_CaseR1}
        \left| \int  R[g_{\eps,t}] w \, d\mu_{g_{\eps,t}} \right| \leq \left( C_0  + C_1 t +C_2 t^2 \right) \|\nabla^{g_{\eps,t}} w\|_{L^2(M, g_{\eps,t})},
    \end{equation}
    where $C_0, C_1, C_2>0$ are constants independent of $t$ and $\eps$.

    Assume that $R[g]=0$, then we have that $R[g] \star_K \rho_\eps = 0$, where the operation $\star_K$ denotes the convolution only on a compact set where the metric is regularised as in the proof of \Cref{prop:NegScalConv}. Applying H\"{o}lder's inequality and the Sobolev inequality, we derive
    \begin{equation}
        \left|\int R[g_\eps] w v \,d\mu_{g_\eps} \right|\leq C(g_\eps)^2 \|R[g_\eps]\|_{L^{n/2}(M, g_\eps)}  \|\nabla^{g_\eps} w\|_{L^2(M, g_\eps)} \|\nabla^{g_\eps} v\|_{L^2(M,g_\eps)},
    \end{equation}
    and,
    \begin{equation}
        \left|\int R[g_\eps] w \,d\mu_{g_\eps} \right|\leq C(g_\eps) \|R[g_\eps]\|_{L^{\frac{2n}{n+2}}(M, g_\eps)}  \|\nabla^{g_\eps} w\|_{L^2(M, g_\eps)},
    \end{equation}
    where $C(g_\eps)$ is the Sobolev constant of $g_\eps$.

    Since $R[g]\star_K \rho_\eps = 0$ and $R[g_\eps] = 0$ outside a compact set $K_0$ containing $K$, we have that $\|R[g_\eps]\|_{L^p(M, g_\eps)} \to 0$ as $\eps \to 0$, for every $p \in [1,n/2]$, by \Cref{prop:NegScalConv}. Then, we obtain
    \begin{align}\label{eq:Scalar_UpperBound1_small}
           \left| \int  R[g_{\eps}] w v \, d\mu_{\eps}\right| \leq  A(\eps)  \|\nabla^{g_\eps} w\|_{L^2(M, g_\eps)} \|\nabla^{g_\eps} v\|_{L^2(M,g_\eps)},
    \end{align}
    and
        \begin{align} \label{eq:Scalar_UpperBound2_small}
           \left| \int  R[g_{\eps}] w \, d\mu_{\eps}\right| \leq  A(\eps)  \|\nabla^{g_\eps} w\|_{L^2(M, g_\eps)},
    \end{align}
    where
    \begin{equation}
         A(\eps)\colonequals C  \|R[g_\eps]\|_{L^{\frac{n}{2}} \cap L^{\frac{2n}{n+2}} (M, g_\eps)} ,
    \end{equation}
    for some constant $C>0$ uniformly in $\eps$, for $\eps$ sufficiently small, and $A(\eps) \to 0$ as $\eps \to 0$.

    Again, combining \eqref{eq:Scalar_UpperBound1_small}, \eqref{eq:Claim1_divterm}, \eqref{eq:Claim1_Lapterm}, \eqref{eq:Claim1_Ricciterm}, and \eqref{eq:Claim1_Bkterm} (or the equivalent estimates for the linear case) and choosing $\eps$ and $t$ smaller if necessary, we have
    \begin{equation}
        \left|\int  R[g_{\eps,t}] w v \, d\mu_{g_{\eps,t}} \right|\leq \left( A(\eps)  + C_1 t +C_2 t^2 \right) \|\nabla^{g_{\eps,t}} w\|_{L^2(M, g_{\eps,t})}\|\nabla^{g_{\eps,t}} v\|_{L^2(M, g_{\eps,t})},
    \end{equation}
    and
    \begin{equation}
        \left|\int  R[g_{\eps,t}] w \, d\mu_{g_{\eps,t}} \right|\leq \left( A(\eps)  + C_1 t +C_2 t^2 \right) \|\nabla^{g_{\eps,t}} w\|_{L^2(M, g_{\eps,t})},
    \end{equation}
    where $C_1,C_2>0$ are constants independent of $t$ and $\eps$ and $A(\eps) \to 0$ as $\eps \to 0$.
    \end{proof}

    \begin{proposition} \label{prop:scalar_flat}
         Let $g \in C^0 \cap \Wloc^{1,n}$ be a complete, asymptotically flat Riemannian metric on $M^n$, smooth outside a compact set, and assume that $R[g] \geq 0 $ in $\D^\prime(M)$. If $m(g) = 0$, then $g$ is scalar-flat.
    \end{proposition}
    \begin{proof}
    If $R[g]\geq 0$ in $\D^\prime$ and $R[g]\neq0$, then there exists a test function $\varphi \in C^\infty_c(M)$, $\varphi\geq0$, such that
    \begin{equation}
        \langle R[g], \varphi\frac{d\mu_g}{d\mu_h} d\mu_h\rangle >0,
    \end{equation}
    where $\frac{d\mu_g}{d\mu_h} \in C^0 \cap W^{1,n}_\loc$ is the density of $d\mu_g$ with respect to $d\mu_h$.

    Denote by $g_\eps$ the regularisation of $g$ via mollifiers on a compact set $K$ (\Cref{lemma:FamilyOfMetrics}). The convergence $g_\eps \to g$ in $C^0_\loc \cap W^{1,n}_\loc$ implies that $\frac{d\mu_{g_\eps}}{d\mu_h} \to \frac{d\mu_{g}}{d\mu_h}$ in $C^0_\loc \cap W^{1,n}_\loc$ as $\eps \to 0$, where $\frac{d\mu_{g_\eps}}{d\mu_h} \in C^0 \cap W^{1,n}_\loc$ is the density of $d\mu_{g_\eps}$ with respect to $d\mu_h$. Together with $R[g_\eps] \to R[g]$ in $\D^\prime$ and to test functions in $C^0 \cap W_\loc^{1,n}$ (also see \Cref{prop:ExtendingRDist}), this implies that
    \begin{equation}
    \int_M R[g_\eps] \varphi d\mu_{g_\eps} = \langle R[g_\eps], \varphi \frac{d\mu_{g_\eps}}{d\mu_h }d\mu_h \rangle \longrightarrow \langle R[g], \varphi\frac{d\mu_{g}}{d\mu_h} d\mu_h\rangle \text{ as } \eps \to 0.
    \end{equation}
    It follows that there exist $\eps_0>0$ and $C_\varphi>0$ such that, for every $\eps \in (0, \eps_0)$, we have
    \begin{equation}\label{eq:ScalarCurv_lowerbound}
    \int R[g_\eps] \varphi d\mu_{g_\eps} \geq C_\varphi>0.
    \end{equation}
    Consider the function,
    \begin{equation}
    f_{\eps,t}\colonequals  R[g_\eps]_+ t\varphi - R[g_\eps]_- =  R[g_\eps]  t\varphi - R[g_\eps]_-(1-t \varphi),
    \end{equation}
    Since $\varphi \geq 0$ and $R[g_\eps]_-$ is compactly supported in a neighbourhood of $K$, it is clear that 
    $$\|(f_{\eps,t})_-\|_{L^{n/2}(M,g_\eps)}=\|R[g_\eps]_-\|_{L^{n/2}(M,g_\eps)} \longrightarrow 0 \quad \text{ as } \eps \to 0,$$
     for all $0 \leq t\leq 1/\|\varphi\|_{L^\infty}$ by \Cref{prop:NegScalConv} and the uniform equivalence of $g_\eps$ and $g$. Then, the integral condition in \Cref{lemma:MiaoConfMetric} is satisfied for $f_{\eps,t}$ for sufficiently small $\eps$ for all $0 \leq t\leq 1/\|\varphi\|_{L^\infty}$. Therefore, there exists $\eps_0>0$ such that for all $\eps \in (0, \eps_0)$ and $t \in [0, 1/\|\varphi\|_{L^\infty}]$, there exists a positive solution $u_{\eps,t} \in C^2$ on $M$ of
    \begin{equation}
        -c_n\Delta_{g_\eps}u_{\eps,t} + f_{\eps,t} u_{\eps,t}  = 0,    
    \end{equation}
    such that $u_{\eps,t} = 1 + A_{\eps,t}/|x|^{n-2} + \omega_{\eps,t}$ outside a compact set, where $A_{\eps,t}$ is a constant and $\omega_{\eps,t} = O(|x|^{1-n})$ and $\partial \omega_{\eps,t} = O(|x|^{-n})$.
    
    The conformally rescaled metric $\widetilde{g}_{\eps,t} = u_{\eps,t}^{\frac{4}{n-2}}g_\eps$ is a $C^2$ asymptotically flat metric and $R[\widetilde{g}_{\eps,t}] $ satisfies
    \begin{equation}
    R[\widetilde{g}_{\eps,t}] = u_{\eps,t}^{-\frac{4}{n-2}} (R[g_\eps]- R[g_\eps]_+ t\varphi + R[g_\eps]_-) = u_{\eps,t}^{-\frac{4}{n-2}} R[g_\eps]_+ (1 - t\varphi )\geq0.
    \end{equation}
    Applying the divergence theorem, the mass formula in \eqref{eq:mass_formula} can also be written as
    \begin{align}\label{eq:estimate_mass}
    m(\widetilde{g}_{\eps,t}) &  =  m(g_\eps) - \frac{1}{2(n-1)\omega_{n-1}} \int_M f_{\eps,t} u_{\eps,t} d\mu_{g_\eps}.
    \end{align}
    In what follows, we derive several estimates for the second term on the right-hand side of \eqref{eq:estimate_mass}. By the distributional scalar curvature lower bound \eqref{eq:ScalarCurv_lowerbound}, we have
    \begin{equation}\label{eq:IntV_eps_t}
    \int_M f_{\eps,t} d\mu_{g_\eps} = \int_M   \left(R[g_\eps]  t\varphi - R[g_\eps]_-(1-t \varphi)\right) d\mu_{g_\eps}
     \geq  t C_\varphi - C \|R[g_\eps]_-\|_{L^{n/2}},    
    \end{equation}      
    where $C$ is a constant independent of $\eps$. 

    Recall that $w_{\eps,t} = u_{\eps,t}-1$. Applying \Cref{lemma:ScalarLowerBound} and the Sobolev inequality \eqref{eq:SobolevConstant}, we can estimate the integral of $f_{\eps,t} w_{\eps,t}$ by
    \begin{align*}
         \left|\int_M f_{\eps,t} w_{\eps,t} d\mu_{g_\eps}\right| & \leq  \left|\int_M   R[g_\eps]  t\varphi w_{\eps,t} d\mu_{g_\eps} \right|  +\left| \int_M R[g_\eps]_-(1-t \varphi)w_{\eps,t} d\mu_{g_\eps}\right|\\
        & \leq  C  \left(t\| \nabla^{g_\eps} \varphi\|_{L^2} \| \nabla^{g_\eps} w_{\eps,t}\|_{L^2} + \|1-t \varphi\|_\infty\|R[g_\eps]_-\|_{L^{\frac{2n}{n+2}}} \|w_{\eps,t}\|_{L^{n^*}} \right)\\
        & \leq  C  \left(t\| \nabla^{g_\eps} \varphi\|_{L^2} \| \nabla^{g_\eps} w_{\eps,t}\|_{L^2} + \|R[g_\eps]_-\|_{L^{n/2}} \| \nabla^{g_\eps} w_{\eps,t}\|_{L^2}\right)\\
         & \leq  C \left(t + \|R[g_\eps]_-\|_{L^{n/2}}\right)\| \nabla^{g_\eps} w_{\eps,t}\|_{L^2},
    \end{align*}
    where $C$ is a constant independent of $\eps$ and $t$.

    Arguing as in the proof of \Cref{prop:ConvGraduetc} and using the above estimate, we obtain
    \begin{align*}
        c_n \int_M |\nabla^{g_\eps} w_{\eps,t}|^2  d\mu_{g_\eps} & \leq  \int_{M} \left((f_{\eps,t})_- w_{\eps,t}^2 - f_{\eps,t} w_{\eps,t} \right) d\mu_{g_\eps} \\
        &  \leq  \|R[g_\eps]_-\|_{L^{n/2}} \|w_{\eps,t}\|_{L^{n^*}}^2  + C_0 \left(t + \|R[g_\eps]_-\|_{L^{n/2}}\right)\| \nabla^{g_\eps} w_{\eps,t}\|_{L^2}\\
        &  \leq  C \|R[g_\eps]_-\|_{L^{n/2}} \| \nabla^{g_\eps} w_{\eps,t}\|_{L^2}^2  +  C_0 \left(t + \|R[g_\eps]_-\|_{L^{n/2}}\right)\| \nabla^{g_\eps} w_{\eps,t}\|_{L^2}.
    \end{align*}
Taking $\eps$ smaller if necessary, $\|R[g_\eps]_-\|_{L^{n/2}}$ is sufficiently small that the term $ \| \nabla^{g_\eps} w_{\eps,t}\|_{L^2}^2$ can be absorbed into the left-hand side. Then
    \begin{equation}\label{eq:gradient estimate_Vepst}
    \| \nabla^{g_\eps} w_{\eps,t}\|_{L^2} \leq C_0 \left(t + \|R[g_\eps]_-\|_{L^{n/2}}\right),
    \end{equation}
    where $C_0$ is a constant independent of $\eps$ and $t$.
    Combining the estimates for $f_{\eps,t}$ with the gradient estimate \eqref{eq:gradient estimate_Vepst} yields
    \begin{equation}
    \int_M f_{\eps,t} u_{\eps,t} d\mu_{g_\eps} \geq  t C_\varphi - C \|R[g_\eps]_-\|_{L^{n/2}} - C_0^2 \left(t + \|R[g_\eps]_-\|_{L^{n/2}}\right)^2.
  \end{equation}
    Then, there exist $\eps_0>0$ and $t_0>0$ such that for all $\eps \in (0, \eps_0)$ and $t \in [0, t_0)$, we have
    \begin{equation}\label{eq:V_eps,t_lowerbound}
    \int_M f_{\eps,t} u_{\eps,t} d\mu_{g_\eps} \geq t \frac{C_\varphi}{2} - 2 C \|R[g_\eps]_-\|_{L^{n/2}},
  \end{equation}
  where the constants are independent of $\eps$ and $t$.

  Therefore, combining the previous estimate for $f_{\eps,t}$ \eqref{eq:V_eps,t_lowerbound} and the estimate for $m(\widetilde{g}_{\eps,t})$ \eqref{eq:estimate_mass}, we obtain 
    
    \begin{align}
   m(\widetilde{g}_{\eps,t})&  \leq m(g_\eps) -\frac{1}{ 2(n-1)\omega_{n-1}}\left(  t \frac{C_\varphi}{2} - 2 C\|R[g_\eps]_-\|_{L^{n/2}} \right),
    \end{align}
    for all $\eps \in (0, \eps_0)$ and $t \in [0, t_0)$. Since $m(g_\eps) =0$ and $\|R[g_\eps]_-\|_{L^{n/2}} \to 0$, we have that for any fixed positive $t \in (0, t_0)$, there exists $\eps_0>0$ such that $m(\widetilde{g}_{\eps,t}) <0$ for all $\eps \in (0, \eps_0)$,  with $R[\widetilde{g}_{\eps,t}] \geq 0$, contradicting the smooth positive mass theorem. Then, we have $R[g]=0$.
    \end{proof}

    \begin{proposition}\label{prop:Ricciflat}
        Let $g \in C^0 \cap \Wloc^{1,n}$ be a complete, asymptotically flat Riemannian metric on $M^n$, smooth outside a compact set, and assume that $R[g] \geq 0 $ in $\D^\prime(M)$. If $m(g) = 0$, then $g$ is Ricci-flat.
    \end{proposition}
    \begin{proof}
     By \Cref{prop:scalar_flat}, we have $R[g]=0$ in $\D^\prime$. The proof of Ricci flatness follows a route similar to that of scalar flatness, but it is more involved. Assume that $\Ric[g]\neq 0$ in $\D^\prime$. Then there exist two vector fields $X$ and $Y$, and a test function $\varphi \in C^\infty_c(M)$ such that 
        \begin{equation}
         \langle \Ric[g](X, Y),  \varphi d\mu_g \rangle  \neq 0.
    \end{equation}
     The argument used for the scalar curvature in \Cref{prop:ExtendingRDist} also shows that $\varphi d\mu_g$ is a valid test density for $\Ric[g]$.

    Taking the symmetrized tensor product of $X$ and $Y$, we can build a smooth $(2,0)$-symmetric tensor field $q_0$ such that locally $q_0^{ij} = \frac{1}{2} (X^i Y^j + Y^i X^j)$ and $\Ric[g](X,Y)= \Ric[g]_{ij}X^iY^j = \Ric[g]_{ij}q_0^{ij}$. Define the $C^0 \cap W^{1,n}_\loc$ regular $(0,2)$-tensor field $q_{ij} = \varphi g_{i\ell} g_{jk} q_0^{\ell k}$, and note that the pairing $ \langle \Ric[g], q \rangle_g$ defines a compactly supported distribution satisfying
    \begin{equation}
         \langle \Ric[g](X, Y),  \varphi d\mu_g \rangle = \langle \langle \Ric[g], q \rangle_g ,  d\mu_g \rangle  \neq 0.
    \end{equation}
    Replacing $q_0$ by $-q_0$ if necessary, we may assume that
        \begin{equation}
        \langle \langle \Ric[g], q\rangle_g, d\mu_g \rangle  < 0.
    \end{equation}

    Let $g_\eps$ be the regularisation of $g$ via mollifiers on a compact set $K$ that also containing the support of $q$ (\Cref{lemma:FamilyOfMetrics}). Define a smooth $(0,2)$-symmetric tensor by the local formula $(q_\eps)_{ij} = \varphi (g_\eps)_{i\ell} (g_\eps)_{jk} q_0^{\ell k}$. For any $\eps>0$ and $t>0$ sufficiently small, define $g_{\eps,t} = g_\eps + t q_\eps$. Since $q_\eps$ is compactly supported, $g_{\eps,t}$ is a compactly supported perturbation of $g_\eps$, and for sufficiently small $\eps$ and $t$, the tensor $g_{\eps,t}$ remains a Riemannian metric uniformly in $\eps$ and $t$ and is uniformly equivalent to $g$. 

    It is not hard to see that $\Ric[g_\eps] \to \Ric[g]$ as distributions for test functions in $C^0_\loc \cap W^{1,n}_\loc$ (see \Cref{prop:ExtendingRDist}), and that $q_\eps \to q$ and $d\mu_{g_\eps}/d\mu_{g} \to 1$ in $C^0_\loc \cap W^{1,n}_\loc$ since $g_\eps \to g$ in $C^0_\loc \cap W^{1,n}_\loc$. Then there exist $\eps_0>0$ and $C_q>0$ such that for all $\eps \in (0, \eps_0)$, we have
    \begin{equation} \label{eq:Riccibound}
        \langle \Ric[g_\eps], q_\eps d\mu_{g_\eps} \rangle = \int_M \langle \Ric[g_\eps], q_\eps\rangle_{g_\eps} d\mu_{g_\eps}= \int_M \Ric[g_\eps]_{ij} q_0^{ij} \varphi d\mu_{g_\eps} \leq -C_q<0,
    \end{equation} 
    Unfortunately, we cannot apply \Cref{lemma:MiaoConfMetric} directly with the functions $f_{\eps,t} = R[g_{\eps,t}]$ to obtain a scalar-flat metric, because the $L^{n/2}$ estimate would be too restrictive to provide uniform control in $t$ and $\eps$. However, \Cref{lemma:ScalarLowerBound} provides a uniform estimate showing that $R[g_{\eps,t}] \to R[g]=0$ in $\D^\prime$ and in a relevant class of functions uniformly as $\eps$ and $t$ tend to zero. We use this estimate to obtain the desired control over the solutions of the conformal Laplacian. More precisely, using this fact, the following claim will be proved below. 
    
    Claim: For any $\delta>0$, there exist $\eps_0, t_0>0$ such that for all $t \in (0,t_0)$ and $\eps \in (0, \eps_0)$, we can find a positive $C^2$ solution $u_{\eps,t}>0$ of 
    \begin{equation} \label{eq:PDE_u}
    -c_n\Delta_{g_{\eps,t}}u_{\eps,t} + R[g_{\eps,t}] u_{\eps,t}  = 0,    
    \end{equation}
    satisfying $\|\nabla^{g_{\eps,t}}u_{\eps,t}\|_{L^2(M, g_{\eps,t})} < \delta$ and having the expansion stated in \Cref{lemma:MiaoConfMetric}.
     
    From this claim, we have that the conformally rescaled metric $\widetilde{g}_{\eps,t} = u^{\frac{4}{n-2}}_{\eps, t} g_{\eps,t}$ is $C^2$ asymptotically flat, scalar-flat, i.e.,
    \begin{equation}
    R[\widetilde{g}_{\eps,t}] = u_{\eps,t}^{-\frac{4}{n-2}} ( R[g_{\eps,t}] - R[g_{\eps,t}])= 0, 
    \end{equation}
    and the ADM mass is given by
    \begin{align}
    2(n-1)\omega_{n-1} m(\widetilde{g}_{\eps,t})&  = 2(n-1)\omega_{n-1} m(g_{\eps,t}) - \int_M R[g_{\eps,t}]  u_{\eps,t}   \,d\mu_{g_{\eps,t}}\\
    &  = - \int_{M}  R[g_{\eps,t}]  u_{\eps,t}  \,d\mu_{g_{\eps,t}},
    \end{align}
    since $m(g_{\eps,t}) = m(g_{\eps}) = m(g) = 0$, and note that $R[g_{\eps,t}] = R[g_{\eps}]=0$ outside a compact set. 

    Let $d\mu_{g_{\eps,t}} / d\mu_{g_{\eps}} \in C^0 \cap W^{1,n}_\loc$ be the density of $d\mu_{g_{\eps,t}}$ with respect to $d\mu_{g_{\eps}}$, and we write
    \begin{equation}
        \int_{M}  R[g_{\eps,t}]  u_{\eps,t}  \,d\mu_{g_{\eps,t}} = \int_{M}  R[g_{\eps,t}]  (u_{\eps,t} \frac{d\mu_{g_{\eps,t}}}{d\mu_{g_{\eps}}} -1 )  \,d\mu_{g_{\eps}} + \int_{M}  R[g_{\eps,t}]  \,d\mu_{g_{\eps}}.
    \end{equation}

    We now estimate $m(\widetilde{g}_{\eps,t})$ in terms of $R[g_{\eps,t}]$. As in the proof of \Cref{lemma:ScalarLowerBound}, recall the following formula \cite[Exercise 1.18]{L:19} (or \cite[Proposition 4]{BrendleMarques2011}),
    \begin{equation} \label{eq:VarScalarCurv_complete}
        R[g_{\eps,t}] = R[g_\eps] + t\restr{DR}{g_\eps}(q_\eps) + t^2 B_{\eps,t}(q_\eps),
    \end{equation}
    where
    \begin{equation}\label{eq:VarScalarCurv}
        \restr{DR}{g_\eps}(q_\eps) = \operatorname{div}_{g_\eps}(\operatorname{div}_{g_\eps}(q_\eps)) - \Delta_{g_\eps} \tr_{g_\eps}(q_\eps)  -  \langle \Ric[g_\eps] , q_\eps \rangle_{g_\eps},
    \end{equation}
    and $B_{\eps,t}(q_\eps)$ is the sum of a compactly supported divergence term, a term quadratic in $\nabla^{g_\eps} q_\eps$ and $\Ric[g_\eps]*q_\eps^2$.

    Since $q_\eps$ is compactly supported, integrating \eqref{eq:VarScalarCurv_complete} and applying the divergence theorem, we have that
    \begin{align}
         \int_{M}  R[g_{\eps,t}] \,d\mu_{g_\eps} & = \int_{M}  R[g_{\eps}]  - t \langle \Ric[g_\eps] , q_\eps \rangle_{g_\eps} + t^2  B_{\eps,t}(q_\eps) \,d\mu_{g_{\eps}}.
    \end{align}

     The second term can be estimated by the Ricci bound \eqref{eq:Riccibound} obtained from the assumption, while the third term has no contribution of the compactly supported divergence term, the quadratic terms can be estimated with $g_\eps \to g$ in $C^0_\loc \cap W^{1,n}_\loc$ for small $t$, and the Ricci term can be estimated with the distributional convergence $\Ric[g_\eps]$ to $\Ric[g]$. Thus,
    \begin{align}
         \int_{M}  R[g_{\eps,t}] \,d\mu_{g_{\eps}} & \geq  \int_{M}  R[g_{\eps}] \,d\mu_{g_{\eps}} + \frac{C_q}{2} t - C_3 t^2, 
    \end{align}
where $C_3>0$ is a constant independent of $\eps$

    By \Cref{lemma:ScalarLowerBound}, where we take $w_{\eps,t} = u_{\eps,t} \frac{d\mu_{g_{\eps,t}}}{d\mu_{g_{\eps}}} -1$, we also obtain the estimate
    \begin{equation}
        \left| \int_{M}  R[g_{\eps,t}]  (u_{\eps,t} \frac{d\mu_{g_{\eps,t}}}{d\mu_{g_{\eps}}} -1 )  d\mu_{g_{\eps}} \right| \leq \left( A(\eps) + C_1 t +C_2 t^2 \right) \|\nabla^{g_{\eps,t}} u_{\eps,t} \|_{L^2(M, g_{\eps,t})},
    \end{equation}
    where $A(\eps) \to 0$ as $\eps \to 0$, and $C_1$ and $C_2$ are positive constants independent of $\eps$ and $t$, for $\eps$ and $t$ sufficiently small (see computations after \eqref{eq:Claim1_Bkterm}). Observe that the lemma can also be applied with $d\mu_\eps$ instead of $d\mu_{\eps,t}$, and the density $\frac{d\mu_{g_{\eps,t}}}{d\mu_{g_{\eps}}}$ was absorbed in the constants as in the proof of \Cref{lemma:ScalarLowerBound}. Although \Cref{lemma:ScalarLowerBound} was proved for fixed $q$, its proof applies uniformly to the family $q_\eps$. In fact, the constants depend on $\|q_\eps\|_{L^\infty}$ and $\|\nabla^{g_\eps} q_{\eps}\|_{L^p}$ for the relevant $p \in[2,n]$. These quantities are uniformly bounded in $\eps$ since $q$ is compactly supported for a fixed compact set, $q_\eps \to q$ in $C^0_\loc \cap W^{1,n}_\loc$ and the metrics $g_\eps$ are uniformly equivalent.

    Combining the previous estimates yields 
     \begin{equation}
        \int_{M} R[g_{\eps,t}] u_{\eps,t} \,d\mu_{g_{\eps,t}} \geq \int_{M}  R[g_{\eps}] \,d\mu_{g_{\eps}}+  \frac{C_q}{2} t - C_3 t^2 - \left( A(\eps)  + C_1 t +C_2 t^2 \right) \|\nabla^{g_{\eps,t}} u_{\eps,t}\|_{L^2},
    \end{equation}

    Therefore, by the claim, for any $\delta>0$, there exist $\eps_0, t_0>0$, such that for any $\eps \in (0, \eps_0)$, $t \in (0,t_0)$, we have
    \begin{equation}
     2(n-1)\omega_{n-1} m(\widetilde{g}_{\eps,t}) \leq -\int_{M}  R[g_{\eps}] \,d\mu_{g_{\eps}}-  \frac{C_q}{4}t + A(\eps)\delta  + (C_1+C_2)t  \delta.
    \end{equation}
    Taking $\delta =\frac{C_q }{8 (C_1+C_2)}$, we have
    \begin{equation}
     2(n-1)\omega_{n-1} m(\widetilde{g}_{\eps,t}) \leq -\int_{M}  R[g_{\eps}] \,d\mu_{g_{\eps}} -   \frac{C_q}{8} t +  A(\eps)  \frac{C_q }{8 (C_1+C_2)}.
    \end{equation}
    Since $R[g] = 0$ and $R[g] \star_M \rho_\eps=0$, we have that $R[g_\eps] \to 0$ in $L^1$ by \Cref{prop:ConvofRs}. Therefore, for any fixed $t \in (0,t_0)$, there exists $\eps_0>0$ such that for all $\eps \in (0, \eps_0)$ we obtain $m(\widetilde{g}_{\eps,t}) < 0$, with $R[\widetilde{g}_{\eps,t}] = 0$, contradicting the smooth positive mass theorem. Hence, $g$ is Ricci-flat.
    
    \medskip
    \noindent \emph{Proof of the Claim.} For this proof, we briefly use the theory of elliptic operators in weighted Sobolev spaces as introduced in \cite{Bartnik1986}. Let $w_{\eps,t} \in W^{2,p}_{-\tau}$, $p>n$ and $\tau>(n-2)/2$, and suppose that $w_{\eps,t} $ satisfies 
    \begin{align}\label{eq:PDE_pointwise}
    -c_n \Delta_{g_{\eps,t}} w_{\eps,t} + R[g_{\eps,t}] w_{\eps,t} = 0.
    \end{align}
    Multiplying by $w_{\eps,t}$ and integrating by parts yields
    \begin{equation}
    c_n \int_M |\nabla^{g_{\eps,t}}w_{\eps,t}|_{g_{\eps,t}}^2 \,d\mu_{g_{\eps,t}} +\int _M R[g_{\eps,t}] w_{\eps,t}^2 d\mu_{g_{\eps,t}}  = 0,
    \end{equation}
    where the boundary term vanishes due to the decay rate of $w_{\eps,t}$.

    Since $R[g]=0$ and $w_{\eps,t}=O_1(|x|^{-\tau})$ by weighted Morrey's inequality and \Cref{lemma:ScalarLowerBound}, there exist $\eps_0, t_0>0$, such that for any $\eps \in (0, \eps_0)$, $t \in (0,t_0)$, we have
    \begin{equation}
    \left| \int_{M} R[g_{\eps,t}] w_{\eps,t}^2 \,d\mu_{g_{\eps,t}} \right| \leq  \beta(\eps,t) \|\nabla^{g_{\eps,t}} w_{\eps,t}\|_{L^2(M, g_{\eps,t})}^2,
    \end{equation}
    where $\beta(\eps,t) \to 0$ as $\eps\to 0 $ and $t \to 0$.

    Therefore, we can find $\eps_0,t_0>0$ such that $\beta(\eps,t) <1$ uniformly in $\eps$ and $t$, and
    \begin{equation}
    0 = c_n \int_M |\nabla^{g_{\eps,t}}w_{\eps,t}|_{g_{\eps,t}}^2 \,d\mu_{g_{\eps,t}} + \int_M R[g_{\eps,t}]w_{\eps,t}^2 \,d\mu_{g_{\eps,t}} \geq  (c_n-\beta(\eps,t)) \|\nabla^{g_{\eps,t}} w_{\eps,t}\|_{L^2(M, g_{\eps,t})}^2,
    \end{equation} 
    which implies $\|\nabla^{g_{\eps,t}} w_{\eps,t}\|_{L^2(M, g_{\eps,t})}=0$ uniformly in $t$ and $\eps$, and then $w_{\eps,t} \equiv 0$ on $M$. Since the operator
    \begin{equation}
    -c_n \Delta_{g_{\eps,t}}  + R[g_{\eps,t}]  \colon W^{2,p}_{-\tau} \longrightarrow L^p_{-\tau-2}
    \end{equation}
    is a compact perturbation of the Laplacian $\Delta_{g_{\eps,t}}$, it has index zero. The argument above shows that its kernel is trivial,  so it is an isomorphism for each $\eps$ and $t$.
    
    As a consequence, there exists $w_{\eps,t}$ such that 
    \begin{equation}
    -c_n \Delta_{g_{\eps,t}} w_{\eps,t} + R[g_{\eps,t}] w_{\eps,t}  = - R[g_{\eps,t}].
    \end{equation}
    Multiplying by $w_{\eps,t}$ and integrating by parts again yields
    \begin{equation}
        c_n \int_M |\nabla^{g_{\eps,t}}w_{\eps,t}|_{g_{\eps,t}}^2 \,d\mu_{g_{\eps,t}} + \int _M R[g_{\eps,t}] w_{\eps,t}^2 \,d\mu_{g_{\eps,t}} =  -\int _M R[g_{\eps,t}] w_{\eps,t} \,d\mu_{g_{\eps,t}}.
    \end{equation}
     Applying \Cref{lemma:ScalarLowerBound} to both terms, we have
     \begin{align}
          c_n \|\nabla^{g_{\eps,t}} w_{\eps,t}\|_{L^2(M, g_{\eps,t})}^2 \leq C \beta(\eps,t) \|\nabla^{g_{\eps,t}} w_{\eps,t}\|_{L^2(M, g_{\eps,t})}\left(1  +  \|\nabla^{g_{\eps,t}} w_{\eps,t}\|_{L^2(M, g_{\eps,t})}\right),
     \end{align}
     for some constant $C>0$ independent of $\eps$ and $t$. Therefore, this implies that there exist $\eps_0, t_0>0$ such that for all $\eps \in (0, \eps_0)$ and $t \in (0,t_0)$, we have
          \begin{align}
          \|\nabla^{g_{\eps,t}} w_{\eps,t}\|_{L^2(M, g_{\eps,t})} \leq C \beta(\eps,t) \longrightarrow 0 \quad \text{ as } \eps \to 0 \text{ and } t \to 0,
     \end{align}
     for some constant $C>0$ independent of $\eps$ and $t$.
     
     Let $E \colonequals \{ x \in M \colon u_{\eps,t}<0\}$, where $E$ is precompact since $u_{\eps,t} \to 1$ at infinity. Denote by $(u_{\eps,t})_-$ the negative part of $u_{\eps,t}$. If $E$ is not empty, a density argument allows us to use $(u_{\eps,t})_- \in C^{0,1}_0$ as a test function in the weak formulation of \eqref{eq:PDE_u} as well as in \Cref{lemma:ScalarLowerBound}, then
     \begin{align}
           \|\nabla^{g_{\eps,t}} u_{\eps,t}\|_{L^2(E, g_{\eps,t})}^2 & = -\frac{1}{c_n} \int _E R[g_{\eps,t}] (u_{\eps,t})_-^2 \,d\mu_{g_{\eps,t}}
             \leq \frac{1}{c_n} \beta(\eps,t) \|\nabla^{g_{\eps,t}} u_{\eps,t}\|_{L^2(E, g_{\eps,t})}^2,
     \end{align}
     which is a contradiction since we can choose $\eps$ and $t$ sufficiently small such that $\frac{1}{c_n} \beta(\eps,t)<1$. Then $u_{\eps,t}\geq 0$. Positivity of $u_{\eps,t}$ follows from the strong maximum principle \cite[Theorem 8.19]{GilbargTrudinger2001}. The $C^2$ regularity is obtained from elliptic regularity, while its asymptotic expansion follows from the fact that $u_{\eps,t}$ is $g$-harmonic outside a compact set. This concludes the proof of the claim.
     \end{proof}

     We recall the definition of $\RCD$ spaces in the context of Riemannian manifolds. In general, this setting is formulated for metric-measure spaces (see \cite{Gigli2018}). A $C^0$ metric together with its induced Lebesgue measure naturally provides a metric-measure space. In this section, we denote by $\Delta$ the Dirichlet Laplacian associated with $g$ and by $D(\Delta)$ its domain (see \Cref{rem:DirichletLap} for a definition).

\begin{definition}[$\RCD$ Space] \label{def:RCDspaces} Let $(M^n, g)$ be a Riemannian manifold with $C^0$ metric $g$, and denote by $\mu_g$ the Lebesgue measure and distance $d_g$ induced by $g$. Let ${\sf K} \in \RR$. Then we say that $(M, d_g, \mu_g)$ is an $\RCD({\sf K}, n)$-space provided:
\begin{enumerate}
    \item $W^{1,2}(M)$ is a Hilbert space.
    \item There exists $C>0$ and $p \in M$ such that $\mu_g\left(B_r(p)\right) \leq e^{C r^2}$ for all $r>0$.
    \item For any $f \in W^{1,2}(M)$ satisfying $|\nabla f|_g \in L^{\infty}(M)$, it admits a Lipschitz representative $\tilde{f}$ with $\mathrm{Lip}(\tilde{f}) \leq \|\nabla f\|_{L^\infty}$.
    \item The Bochner inequality is satisfied, i.e.,
        $$
        \frac{1}{2}\int \Delta \varphi |\nabla f|^2 d\mu_g \geq \frac{1}{n} \int_M (\Delta f)^2 \varphi d\mu_g + \int_M\varphi\left[ g(\nabla f, \nabla \Delta f)+ {\sf K} |\nabla f|^2\right] \mathrm{d} \mu_g
        $$
for every choice of functions $f \in D(\Delta)$ and $\varphi \in D(\Delta) \cap L^{\infty}(M)$, $\varphi\geq 0 $, with $\Delta f \in$ $W^{1,2}(M)$ and $\Delta \varphi \in L^{\infty}(M)$.
\end{enumerate}
\end{definition}
We say that $(M^n, g)$ has \textit{nonnegative Ricci curvature in $\RCD$ sense} if $(M^n, d_g, \mu_g)$ is a $\RCD(0,n)$ space.

\begin{remark}\label{rmk:RCD}
Given a Riemannian manifold $(M, g)$ with a $C^0$ metric and Lebesgue measure $\mu_g$, properties (1) and (3) from \Cref{def:RCDspaces} are automatically satisfied. Furthermore, if $(M, g)$ is asymptotically flat, property (2) also holds. If $(M, g)$ is smooth with Ricci curvature bounded below by ${\sf K}$, then property (4) is satisfied as well. Therefore, our goal is to show that, if $m(g)\equiv 0$, the asymptotically flat manifold $(M^n, g)$ satisfies the condition (4) with ${\sf K}=0$.
\end{remark}

\begin{remark}\label{rem:DirichletLap}
    Let $(M, g)$ be a Riemannian manifold with $g \in C^0$. For any $ u \in W^{1,2}(M)$, if there exists $ v \in L^2(M)$ such that
    $$
    \int_M v w \,\dg = - \int_M \langle du, dw \rangle_g\,\dg, \quad \forall w \in C^\infty_c(M), $$
    then we define $\Delta_g u := v$, the Dirichlet Laplacian of $u$. The domain of the Dirichlet Laplacian is 
    $$
    D(\Delta_g):= \{ u \in W^{1,2}(M) \colon u {\text{ has a Dirichlet Laplacian in } L^{2}(M)}\}.
    $$
\end{remark}

    \begin{theorem}\label{theo:PMT_rigidity}
        Let $M^n$, $n\geq3$, be a smooth manifold endowed with a complete, asymptotically flat Riemannian metric $g \in C^0\cap \Wloc^{1,n}$, smooth outside a compact set. If $R[g]\geq0$ in $\D^\prime$ and $m(g)\equiv0$ then $(M^n, d_g)$ and $(\RR^n, d_\delta)$ are isometric as metric spaces.
    \end{theorem}
\begin{proof}
    The theorem now follows analogously to \cite[Theorem 1.1]{Jiang2022} using $\RCD$-spaces by volume stability theorem of $\RCD$ spaces (\cite{Lott2009-la},  \cite{Sturm2006-bs}, see also \cite[Corollary 1.7]{Gigli2018}). We include it for completeness.
    
    \Cref{prop:Ricciflat} implies that $\Ric[g]=0$. Hence, by the equivalence between distributional Ricci curvature lower bounds and $\RCD(\sf{K},n)$-spaces in \cite[Theorem 7.2]{mondino2025}, which essentially establishes property $(4)$ in \Cref{def:RCDspaces}, we have that $(M, d_g, \mu_g)$ is a $\RCD(0,n)$ space, i.e., $(M, g)$ has nonnegative Ricci curvature in $\RCD$ sense. The volume condition in \cite[Theorem 7.2]{mondino2025} is satisfied as discussed in \Cref{rmk:RCD} (also see proof of \cite[Proposition 7.5]{mondino2025}). 

    Let $p \in M$. The asymptotic flatness of $(M, g)$ and $C^0$ regularity of $g$ show that the volume growth is asymptotically Euclidean. Moreover, since $(M, g)$ has nonnegative Ricci curvature in $\RCD$ sense, by the volume comparison for $\RCD$ spaces (see \cite[Corollary 1.7]{Gigli2018})), we have $\Vol(B_r(p)) \leq \omega_n r^n$ for all $r>0$. Combining these two results, we conclude that
\begin{equation}
    \lim_{r \to \infty} \frac{\Vol(B_r(p))}{\omega_n r^n} = 1.
\end{equation}
    For a more detailed argument, see \cite[Lemma 2.6]{Li2018Mass}. Applying the Bishop-Gromov inequality for $\RCD$-spaces with nonnegative Ricci curvature, we have
    $$
 \Vol(B_r(p))= \omega_n r^n.
    $$
    From a volume rigidity result \cite[Corollary 1.7]{Gigli2018}, we obtain that $B_{r/2}(p)$ is isometric to $B_{r/2}(0_{\RR^n}) \subset \RR^n$, which implies that $M$ is isometric to $\RR^n$ in the sense of metric spaces.  
\end{proof}

By a low-regularity version of Myers-Steenrod and assuming $g \in C^0\cap W^{1,p}$, $p>n$, the metric-space isometry is also a Riemannian isometry. 

\begin{corollary} \label{cor:RigidityW1p}
    Let $M^n$ be a smooth manifold endowed with a complete, asymptotically flat Riemannian metric $g \in C^0\cap \Wloc^{1,p}$, $p>n$, smooth outside a compact set. If $R[g]\geq0$ in $\D^\prime$ and $m(g)\equiv0$ then there exists a $C^{1,1-\frac{n}{p}}$ isometry between $(M, g)$ and $(\RR^n, \delta)$.
    \end{corollary}
    \begin{proof}
        \Cref{theo:PMT_rigidity} implies that $(M, g)$ and $(\RR^n, \delta)$ are isometric as metric spaces. Morrey's inequality gives $g \in C^{0,\alpha}$,  where $\alpha=1-n/p$, and the low-regularity version of Myers-Steenrod, \cite[Corollary C]{Matveev2017}, implies that there exists a $C^{1,\alpha}$ isometry from $(M, g)$ to $(\RR^n, \delta)$.  
    \end{proof}

Finally, \Cref{theo:PMT} follows directly from \Cref{theo:PMT_ineq}, \Cref{theo:PMT_rigidity}, and \Cref{cor:RigidityW1p}. 

\section*{Data availability statement}
There is no data associated with this manuscript.

\section*{Conflict of interest statement} The authors have no conflicts of interest to 

\printbibliography[heading=bibintoc]

\end{document}